\documentclass[11pt]{amsart}
\usepackage[utf8]{inputenc}

\usepackage{overpic, amssymb, times, graphicx, tikz-cd,float, wrapfig}
\usepackage[all]{xy}
\usepackage[backref=page]{hyperref}
\usepackage{verbatim}
\usepackage{cleveref}

\hypersetup{
    colorlinks=true,
    linkcolor=blue,
    filecolor=blue,      
    urlcolor=blue,
    citecolor=blue,
}
\newcommand{\necklength}{\mathfrak{nl}}
\DeclareMathOperator{\Glue}{Gl}
\DeclareMathOperator{\Vertex}{Ve}
\DeclareMathOperator{\vertex}{ve}
\DeclareMathOperator{\Edge}{Ed}
\DeclareMathOperator{\edge}{ed}

\DeclareMathOperator{\rank}{rank}
\DeclareMathOperator{\Sec}{Sec}

\DeclareMathOperator{\Map}{Map}
\DeclareMathOperator{\hot}{hot}
\DeclareMathOperator{\Crit}{Crit}
\DeclareMathOperator{\Morse}{Mo}
\DeclareMathOperator{\wind}{wind}

\DeclareMathOperator{\Span}{span}

\DeclareMathOperator{\im}{im}
\DeclareMathOperator{\dom}{dom}

\DeclareMathOperator{\ind}{ind}

\DeclareMathOperator{\Symp}{Symp}

\DeclareMathOperator{\Id}{Id}

\DeclareMathOperator{\CZ}{CZ}

\DeclareMathOperator{\Flow}{Flow}

\DeclareMathOperator{\sgn}{sgn}

\DeclareMathOperator{\Sobalev}{W}
\DeclareMathOperator{\Ltwo}{L^{2}}
\DeclareMathOperator{\rd}{d}

\newcommand{\C}{\mathbb{C}}
\newcommand{\R}{\mathbb{R}}
\newcommand{\Z}{\mathbb{Z}}

\newcommand{\Q}{\mathbb{Q}}
\newcommand{\fX}{\mathfrak{X}}
\newcommand{\disk}{\mathbb{D}}
\newcommand{\grad}{\nabla}

\newcommand{\delbar}{\overline{\partial}}
\newcommand{\Cinfty}{\mathcal{C}^{\infty}}

\newcommand{\sphere}{\mathbb{S}}
\newcommand{\Circle}{\sphere^{1}}
\newcommand{\torus}{\mathbb{T}}

\newcommand{\half}{\frac{1}{2}}

\newcommand{\ModSpace}{\mathcal{M}}

\newcommand{\multisec}{\mathfrak{s}}

\newcommand{\tree}{\thicc{t}}

\newcommand{\be}{\begin{enumerate}}
\newcommand{\ee}{\end{enumerate}}

\newcommand{\Mxi}{(M,\xi)}

\newcommand{\norm}[1]{\left\lVert#1\right\rVert}

\newcommand{\Dlinearized}{\mathbf{D}}
\newcommand{\AsymptoticOp}{\mathbf{A}}
\newcommand{\domainJ}{\mathbf{j}}
\newcommand{\compactSubset}{\thicc{K}}
\newcommand{\NcompactSubset}{\thicc{N}}
\newcommand{\transSubbundle}{\thicc{E}}
\newcommand{\ModSpaceThicc}{\thicc{V}}

\newcommand{\BO}{\mathrm{BO}}
\newcommand{\Openbook}{\mathrm{OB}}
\newcommand{\DG}{\mathrm{DG}}

\newcommand{\thicc}[1]{\pmb{#1}}
\newcommand{\divSet}{\Gamma}

\newcommand{\hypersurface}{S}

\newcommand{\orbit}{\gamma}
\newcommand{\aug}{\thicc{\epsilon}}
\newcommand{\framing}{\mathfrak{f}}
\newcommand{\eigenvalue}{a}
\newcommand{\eigenfunction}{\eta}
\newcommand{\eigenspace}{A}
\newcommand{\eigenbound}{\thicc{\eigenvalue}}
\newcommand{\foliation}{\mathcal{F}}
\newcommand{\leaf}{\mathcal{L}}
\newcommand{\leafInclusion}{\mathcal{I}}
\newcommand{\completion}{\widehat} 

\newtheorem{thm}{Theorem}[section]
\newtheorem{theorem}[thm]{Theorem}

\newtheorem{ex}[thm]{Example}

\newtheorem{prop}[thm]{Proposition}

\newtheorem{defn}[thm]{Definition}

\newtheorem{lemma}[thm]{Lemma}
\newtheorem{cor}[thm]{Corollary}

\newtheorem{rmk}[thm]{Remark}

\newtheorem{notation}[thm]{Notation}

\title{Bourgeois' contact manifolds are tight}
\author{Russell Avdek and Zhengyi Zhou}
\date{\today}

\begin{document}

\begin{abstract}
We prove that Bourgeois' contact structures on $M \times \torus^{2}$ determined by the supporting open books of a contact manifold $\Mxi$ are always tight. The proof is based on a contact homology computation leveraging holomorphic foliations and Kuranishi structures.
\end{abstract}

\maketitle
\numberwithin{equation}{subsection}
\setcounter{tocdepth}{1}
\tableofcontents

\section{Introduction}

Given a contact open book decomposition of a contact manifold $\Mxi$, Bourgeois \cite{Bourgeois02} showed that $M\times \torus^{2}$ carries a natural contact structure, which is now referred to as a \emph{Bourgeois contact manifold}. One motivation behind such a construction was the problem of the existence of contact structures on manifolds of dimension $\geq 5$. By the Giroux correspondence \cite{BHH, Giroux02} any $\Mxi$ has a supporting open book decomposition, hence Bourgeois' construction can be applied to endow $M\times \torus^{2}$ a contact structure whenever $M$ is contact\footnote{We emphasize that Bourgeois' construction depends on the supporting open book, not just the contact manifold.}. For instance, Bougeois' construction shows that every odd dimensional torus admits contact structures. The existence of contact structures was solved in full generality by Borman-Eliashberg-Murphy \cite{BEM}: They establish that the existence of contact structures is equivalent to the existence of almost contact structures -- a purely topological notion --  and moreover a full $h$-principle for overtwisted contact structures.

With the existence of contact structures established, it is natural to investigate the existence of tight (that is, not overtwisted) contact structures. In \cite{Fabio:Constructions}, Gironella establishes that when $\Mxi$ is weakly fillable, then so is $M \times \torus^{2}$, implying tightness \cite{MNW12}. Bowden-Gironella-Moreno \cite{BGM} proved that all $5$-dimensional Bourgeois contact manifolds $M^{3} \times \torus^{2}$ are tight. In this paper we establish that Bourgeois' construction always yields a tight contact structure, in any dimension, even if we start with $\Mxi$ an overtwisted manifold.

\begin{theorem}\label{thm:main}
    Bourgeois' contact structures are always tight. 
\end{theorem}

\Cref{thm:main} answers Lisi-Marinkovi{\'c}- Niederkr{\"u}ger's question \cite[Question 1.2.(a)]{LMN} (also \cite[Question 32]{BGM}) in the negative. Since each $M \mapsto M \times \{ pt\} \subset M\times \torus^2$ is a contact embedding of $\Mxi$ with trivial normal bundle, the following foundational result of Hern\'andez-Corbato, Mart\'in-Merch\'an and Presas \cite{tightneighorhood} is esablished as an immediate consequence of Theorem \ref{thm:main}. 

\begin{cor}
For any contact manifold $(M,\xi=\ker \alpha)$, there exists $\epsilon>0$ such that $(M\times \disk_{\epsilon}, \ker(\alpha+r^2\rd \theta))$ is tight. 
\end{cor}

\subsection{Outline}\label{Sec:IntroOutline}

Here we provide a brief overview of the proof of Theorem \ref{thm:main} and of the remainder of the article, starting with some generalities and a review of previous result.

To prove tightness of contact manifolds of dimension $\geq 5$, the only currently available tool is the non-vanishing of contact homology, $CH$ \cite{BH, Pardon:Contact}. A combination of Bourgeois-van Koert's \cite{BvK} and Casals-Murphy-Presas' \cite{CMP} shows that contact homology vanishes for any overtwisted contact manifold. So if contact homology is non-vanishing, then the underlying contact manifold is tight. 

The criterion has many avatars in applications, e.g.\ the existence of symplectic fillings, the hyper-tight property, properties on the Conley-Zehnder indices, etc. For example, in the dimension $3$, all Bourgeois contact manifolds are weakly fillable, and Bowden-Gironella-Moreno's proof in dimension $5$ used both fillings and the hyper-tight property.\footnote{A hyper-tight contact manifold is one with a Reeb vector field having no contractible Reeb orbits, guaranteeing the non-vanishing of $CH$.} Their proof relied on properties of mapping class groups of surfaces, as the monodromy in the mapping class group is the only data needed to build a $5$-dimensional Bourgeois contact manifold. However, this is difficult to generalize to higher dimensions as comparatively little is known about symplectic mapping class groups in higher dimensions. Moreover as illustrated in \cite{BGM,bowden2022non},  Bourgeois contact manifolds tend to not be (strongly) fillable in high dimensions.

We prove \Cref{thm:main} by a direct computation of the contact homology of a covering $M \times \Circle \times \R_{\tau} $ of the Bourgeois contact manifold, which is an infinite neighborhood of a ``Bourgeois convex hypersurface'' $\hypersurface = M\times \Circle$ (with $\R_{\tau}$ invariant contact structure). Specifically, we use the sutured contact homology of \cite{CGHH} which can be applied to non-compact manifolds such as $\R_{\tau} \times \hypersurface$. We show that $M \times \Circle\times \R_{\tau}$ has non-zero contact homology, from which it follows that this non-compact manifold is tight. Since $M \times \torus^{2}$ has a tight covering, it then follows that $M \times \torus^{2}$ is tight.

Our computation of the contact homology of $M \times \Circle\times \R_{\tau}$ is based on \cite{Avdek:Hypersurface}, which computes the contact homology of neighborhoods of convex hypersurfaces in general. Here we apply similar techniques to those of \cite{Avdek:Hypersurface}, leveraging holomorphic foliations and Kuranishi structures. Throughout, we use Bao-Honda's formulation of $CH$ \cite{BH} using a Kuranishi perturbation scheme and expect that an application of Pardon's perturbation scheme \cite{Pardon:Contact} would yield an identical result. We note that the proof of the vanishing of $CH$ for overtwisted contact manifolds in \cite{BvK} is independent of perturbation schemes and so the results hold for any formulation of $CH$.

In \S \ref{Sec:Bourgeois} we provide an overview of Bourgeois construction and provide a decomposition of the convex hypersurface $\hypersurface = M \times \Circle$ into positive and negative regions. When $\Mxi$ is presented as an open book with page $V$ and monodromy $\phi$, these positive and negative regions are each copies of $\completion{V} \times T^{\ast}\Circle$ and their boundaries are identified with a contactomorphism determined by $\phi$. Here and throughout, $\completion{V}$ denotes the completion of a Liouville manifold $V$.

In \S \ref{Sec:Foliation} through \S \ref{Sec:VanishingAug}, we study holomorphic curves in complete Liouville manifolds of the form $\completion{W}=\completion{V} \times \completion{\Sigma}$ where $\Sigma$ is a non-simply connected Riemann surface. This generalizes the case $\completion{\Sigma} = T^{\ast}\Circle$ relevant to Bourgeois contact manifolds as mentioned above. The analysis in these sections culminates in Theorem \ref{thm:no_augmentation}, which shows that with a specific choice of contact form on its ideal boundary $\divSet = \partial W$, almost complex structures and perturbation data, the augmentation $\aug_{W}$ associated to this filling of $\divSet$ is \emph{trivial} in the sense that $\aug_{W}\orbit=0$ for each Reeb orbit generator $\orbit$ of the chain-level contact homology algebra of $\divSet$.

Finally in \S \ref{Sec:NonVanishing}, we establish that the augmentation associated to the negative region of the convex hypersurface is DG homotopic to a trivial augmentation. This is \emph{not} tautological, as from the perspective of the negative region, the boundary contact form is different from the positive region and they are related by a \emph{strict contactomorphism}, i.e.\ a diffeomorphism preserving contact forms. The properties of these augmentations combine the with the Algebraic Giroux Criterion of \cite{Avdek:Hypersurface} to establish that $M \times \Circle \times \R_{\tau}$ has non-vanishing contact homology.

\subsection*{Acknowledgments}
\begin{wrapfigure}{r}{0.07\textwidth}
	\centering
	\includegraphics[width=.07\textwidth]{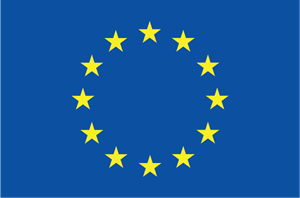}
\end{wrapfigure}

R.A. thanks Cofund MathInGreaterParis and the Fondation Mathématique Jacques Hadamard for supporting him as a member of the Laboratoire de Math\'{e}matiques d'Orsay at Universit\'{e} Paris-Saclay. This project has received funding from the European Union’s Horizon 2020 research and innovation programme under the Marie Skłodowska-Curie grant agreement No 101034255. Z.Z. is supported by the National Key R\&D Program of China under Grant No.\ 2023YFA1010500, the National Natural Science Foundation of China under Grant No.\ 12288201 and 12231010.

\section{Bourgeois contact manifolds}\label{Sec:Bourgeois}

\subsection{Bourgeois contact manifolds from contact open books}

Let $(V,\beta_{V})$ be a Liouville domain and $\phi\in \pi_0(Symp_c(V, \beta_{V}))$ be a compactly supported symplectomorphism. Then $\phi^*\beta_{V}=\beta_{V} + \eta$, where $\eta$ is a closed $1$-form which is zero near $\partial V$. The Thurston-Winkelnkemper construction \cite{TW:OpenBook} gives the total space $M=(\disk \times \partial V) \cup V_{\phi}$ of the open book a contact form $\alpha$, where $V_{\phi}=V \times [0,2\pi]_\theta/(x,2\pi)\sim (\phi(x),0)$ is the mapping torus. More precisely,
\begin{equation}\label{eqn:contact_OB}
    \alpha = \begin{cases}
    \beta_{V}|_{\partial V} + Kr^2 \rd \theta & \text{along}\ \disk \times \partial \Sigma, \\
    \beta_{V} + b(\theta)\eta+K \rd \theta & \text{along} \Sigma_{\phi}
    \end{cases}
\end{equation}
for $K\gg 0$ and $b=b(\theta)$ a function equalling $1$ at $\theta=2\pi$ and $0$ at $\theta=0$. Such a contact open book is denoted by
\begin{equation*}
    \Mxi = \Openbook(V, \phi)
\end{equation*}
Strictly speaking, the contact open book has a corner $\Circle \times \partial V$, since its neighborhood with contact form is identified with the neighborhood of the corner in the boundary of the product $(\disk \times V, Kr^2\rd \theta+\lambda)$. One can round it following e.g. \cite[\S 2.1]{Zhou:DiskTimesV}.

Let $\rho: M \rightarrow [0, 1]$ agreeing with $r$ near $\{0\} \times \partial V$ in $\disk \times \partial V$, equalling $1$ along the boundary of this neighborhood, and equalling $1$ elsewhere. Then the functions $\Phi_{1}=\rho\cos(\theta), \Phi_{2}=-\rho\sin(\theta)$ are defined on all of $M$, so that $\Phi=(\Phi_{1}, \Phi_{2}): M \rightarrow \C$ has $0$ as a regular value with inverse image $\partial V$, which is the binding of the open book for $\Mxi$.

\begin{theorem}[Bourgeois \cite{Bourgeois02}] \label{thm:bourgeois}
	 $\alpha_{\mathrm{BO}} := \alpha + \Phi_1 \rd x + \Phi_2 \rd y$ is a contact form on $M \times \torus^{2} = \Openbook(V,\phi)\times \torus^2$, where $(x,y)$ are coordinates on $\torus^2$. The associated contact structure is independent of choices (eg. $K$, $b$, and the isotopy class of $\phi$).
\end{theorem}

We refer to such contact manifolds as \emph{Bourgeois contact manifolds}, denoted by $\BO(V,\phi)$. Observe that our contact form (and hence contact structure) is $\torus^{2}$ invariant. In particular, each $M \times \{ x\} \times \Circle_{y}$ is a convex hypersurface with respect to the contact vector field $\partial_{x}$.

\subsection{$\Circle$ invariant contact structures}

In \cite{DingGeiges:S1},  Ding and Geiges showed that any $\Circle$ invariant contact form $\alpha$ on $\Circle_{\tau} \times \hypersurface$ admits a decomposition 
$\Circle_{\tau}\times (W_+\cup_{\Gamma} -W_-)$ with 
\begin{equation*}
\alpha=f\rd \tau +\beta, \quad f\in \Cinfty(\hypersurface), \quad \beta \in \Omega^{1}(\hypersurface)
\end{equation*}
such that on the interior $W^{\circ}_{\pm}$, $\pm f>0$ and $\beta_{\pm} = \pm \beta/f \in \Omega^{1}(W_{\pm}^{\circ})$ define Liouville forms. The symplectic forms $d\beta_{\pm}$ then give the $W_{+}$ ($W_{-}$) the same (the opposite) orientation as is determined by the orientation of $\hypersurface$. It also follows that $\alpha$ is a contact form when restricted to $\divSet=\{f=0, \tau=\tau_{0}\}$ and we write $\xi_{\divSet} = \ker \alpha|_{\divSet}$. With this contact structure, $(\divSet, \xi_{\divSet})$ is the ideal boundary of the $(W_{\pm}, \beta_{\pm})$. In other words, $(W_{\pm}, \beta_{\pm},\ker \beta|_{\Gamma})$ are ideal Liouville domains introduced by Giroux \cite{Giroux20} in the following sense:

\begin{defn}[{\cite[\S 4.2]{MNW12}}]
Let $W$ be a compact $2n$-dimensional manifold with boundary, $\omega$ a symplectic form on the interior $W^\circ$ of $W$ and $\xi$ a contact structure on $\partial X$. The triple $(W, \omega, \xi)$ is an ideal Liouville domain if there exists an auxiliary 1-form $\beta$ on $W^\circ$ such that:
\begin{enumerate}
    \item $\rd \beta = \omega$ on $W^{\circ}$;
    \item For any (hence every) smooth function $f:W \to [0, \infty)$ with regular level set $\partial U= f^{-1}(0)$, the 1-form $f\beta$ extends smoothly to $\partial U$ such that its restriction to $\partial W$ is a contact form for $\xi$.
\end{enumerate}
In this situation, $\beta$ is called a Liouville form for $(W, \omega, \xi)$.
\end{defn}

Here we list some basic properties of ideal Liouville domains.

\begin{enumerate}
    \item The contact structure $\xi$ is determined by $\omega$, \cite[Proposition 2]{Giroux20}.
    \item The space of auxiliary data, i.e.\ $f,\beta$, is contractible. See \cite[Remark 4.3]{MNW12}.
    \item Given a Liouville form $\beta$, the Liouville vector $X$, i.e.\ the $\omega$ dual of $\beta$, is a complete vector field on $W^{\circ}$. Given a function $f$ as in the above definition, $\frac{1}{f}X$ extends smoothly over $\partial W$ and points outward \cite[Lemma 4.5]{MNW12}. More precisely, let $\gamma=f\beta$, which extends to $W$ by definition and $\mu=f^{n+1}\omega^n=f(\rd \gamma)^n-n\rd f \wedge \gamma \wedge (\rd \gamma)^{n-1}$. Then $\frac{1}{f}X$ is characterized by 
    \begin{equation}\label{eqn:X_f}
        \iota_{\frac{1}{f}X} \mu = n\beta\wedge (\rd \beta)^{n-1}.
    \end{equation}
    \item Ideal Liouville domains are stable, in the sense that paths of ideal Liouville structures yields homotopic Liouville structures on $W^{\circ}$ \cite[Lemma 6]{Giroux20}.
    \item Given a Liouville domain $(W,\beta_{W})$, $(\overline{W},\rd(f\beta_{W}),\ker(\beta_{W}|_{\partial W}))$ is an ideal Liouville domain for some positive function $f$ on $W$. Such correspondence is one-to-one up to homotopies of (ideal) Liouville domains.
\end{enumerate}

In the case of $\Circle$ invariant contact structures on $\Circle \times \hypersurface$, $\frac{1}{f}W_+$ and $\frac{1}{f}W_-$ patch together smoothly on $\hypersurface$ and are transverse to $\Gamma$ from $W_{+}$ to $W_{-}$, where $X_{\pm}$ is the Liouville vector field for the two ideal Liouville domains. This follows from Equation \eqref{eqn:X_f}.

Conversely, let $(W_+,\omega_+,\xi_+)$ and $(W_-,\omega_-,\xi_-)$ be two ideal Liouville domains. Given a contactomorphism $\psi:(\partial W_+,\xi_+)\to (\partial W_-,\xi_-)$, then we can build a $\Circle$-invariant contact structure as follows. Choose a contact form $\alpha_+$ for $(\partial W_+,\xi_+)$ such that $\omega_+=\rd \beta_+$ with $f_+\beta_+|_{\partial V_+}=\alpha_+$. Let $\alpha_-=\psi_*\alpha_+$, we find $\beta_-$ and $f_-$, such that $\omega_-=\rd \beta_-$ and $\alpha_-=f_-\beta_-|_{\partial V_-}$. By \cite[Lemma 4.5]{MNW12}, we can choose the $f_{\pm}$ such that there is a collar neighborhood of $\partial W_{\pm} \subset W_{\pm}$ modeled on $(0,1]_s\times \partial W_{\pm}$ with $f_{\pm}=1-s$ and $\beta_{\pm}=\frac{1}{1-s}\alpha_{\pm}$. With such $f_{\pm}$, $\Circle_{\tau} \times (V_+\cup_{\psi} -V_-)$ with $\alpha=\left(f_+\cup (-f_-)\right) \rd \tau + f_+\beta_+\cup f_-\beta_-$ is an $\Circle$-invariant structure. 

\begin{prop}
    The above $\Circle$-contact structure is well-defined, up to the homotopy of the defining data, i.e.\ ideal Liouville structures and contactomorphism.
\end{prop}

\begin{proof}
By stability of ideal Liouville domains and the contractibility of auxiliary choices, homotopies of the defining data give rise to families of $\Circle$-invariant contact structures via the construction above. So the claim follows from the Gray's stability theorem.
\end{proof}

\begin{defn}
We write $\DG(W_+,W_-,\psi)$ for the $\Circle$-invariant contact manifold determined by the above construction. 
\end{defn}

\begin{ex}\label{ex:trivial}
    Let $(V,\beta_{V})$ be a Liouville domain. Then $(\partial(V\times [-1,1]_{p} \times\Circle_{q}), \alpha:=\beta+p\rd q)$ is an $\Circle$-invariant contact manifold, where the $\Circle$ action is given by the rotation in $q\in \Circle$ coordinate. Strictly speaking, we need to round the corner, which clearly will not affect the discussion. Hence $f=\alpha(\partial_q)=p$, which is a smooth function after rounding the corner. Then we have that $(\partial(V\times \disk^*\Circle), \alpha)=\DG(V,V,\Id)$.
\end{ex}

Returning to the case of Bourgeois contact manifolds $\Openbook(V,\phi)\times \Circle_x\times \Circle_y$, we can view it as an $\Circle_y$ invariant contact form. We first start with the trivial case:

\begin{ex}\label{ex:trivial_BO}
We consider $\BO(V,\phi)$, which is equipped with an $\Circle_y$-action. The fact that $\BO(V,\Id)=\partial (V \times \disk^*\Circle_x\times \disk^*\Circle_y)$ was established in \cite[Theorem A(b)]{LMN} by explicit computations. Here we give an alternative approach using abstract homotopies. Such perspective will be helpful for the general case of \Cref{prop:BO} appearing below.

Note that $f=\alpha_{\mathrm{BO}}(\partial_y)=\Phi_2$. Therefore the dividing set $\Gamma$ is given by 
\begin{equation*}
    \left(\left(V \times \{(-1,0)_{a,b}\} \cup V \times \{(1,0)_{a,b}\}\cup \partial V \times  [-1,1]_a \times \{0\}_b \right)\times \Circle_x, \beta_{V}+a\rd x)\right)
\end{equation*}
where $(a,b)=(\Phi_1,\Phi_2)$ are the coordinates on $\disk$. That is $\divSet$ is $\partial(V \times \disk^*\Circle)$. The $V_+$ piece, as an ideal Liouville domain, is given by 
\begin{equation*}
    \left(\left(V \times \partial \disk_{b>0} \cup \partial V \times \disk_{b>0}\right)\times \Circle_x,\beta_K:=\frac{\lambda+Ka \rd b-Kb\rd a+a\rd x}{b}\right)
\end{equation*}
where $\disk_{b>0}$ is the disk in the $b>0$ half plane and $\partial \disk_{b>0}$ is the arc in the $b>0$ half plane. By tuning down $K$, it is clear the contact form on the ideal boundary does not change. Moreover, in the trivial open case, the Thurston-Winkelnkemper construction \cite{TW:OpenBook} as well as the Bourgeois construction works for any $K>0$, i.e.\ $\rd \beta_K$ is symplectic for any $K>0$. When $K=0$, it is direct to check that $\rd \beta_0$ is symplectic\footnote{It is easier to check using $(\frac{a}{b},b)$ as coordinate on $\disk_{b>0}$.}. Therefore, $W_+$ is homotopic (and contact structure on the ideal boundary is strictly fixed in this homotopy) to 
\begin{equation*}
    \left(\left(V \times \partial \disk_{b>0} \cup \partial V \times \disk_{b>0}\right)\times \Circle_x,\beta_0:=\frac{\beta_{V}+a\rd x}{b}\right).
\end{equation*}
Similarly, the $W_-$ piece is homotopic to 
\begin{equation*}
    \left(\left(V \times \partial \disk_{b<0} \cup \partial V \times \disk_{b<0}\right)\times \Circle_x,-\frac{\beta_{V}+a\rd x}{b}\right).
\end{equation*}
Now note that $(\partial(V\times \disk \times \Circle_x\times \Circle_y), \beta_{V}+a\rd x+b\rd y)$ is also an $\Circle_y$-invariant contact form, with the same $f$ and dividing set $\divSet$ as $\BO(\Sigma,\Id)$. And the associated $W_{\pm}$ pieces are the above ideal Liouville domains after homotopy. Hence  we have $\BO(V,\Id)$ is homotopic to $(\partial(V\times \disk \times \Circle_x\times \Circle_y), \ker(\beta_{V}+a\rd x+b\rd y))$ as $\Circle$-invariant contact manifolds. Finally, we deform $\disk$ to $[-1,1]^2$ in $\R^2$ in an obvious way. Such deformation yields a family of $\Circle_y$-invariant contact manifolds with the same dividing set, and the $W_{\pm}$ are deformation equivalent ideal Liouville domains. Therefore we have $\BO(V,\Id)$ is homotopic to $\partial (V \times \disk^*\Circle_x\times \disk^*\Circle_y)=\partial(V\times [-1,1]^2 \times \Circle_x\times \Circle_y)$ as $\Circle$-invariant contact manifolds. By \Cref{ex:trivial}, the $W_{\pm}$ pieces of $\partial (V \times \disk^*\Circle_x\times \disk^*\Circle_y)$ are $V \times \disk^*\Circle_x$, hence so are $W_{\pm}$ pieces of $\BO(V,\Id)$, i.e.\ $\BO(V,\Id)=\DG(V \times \disk^*\Circle,\Id)$.
\end{ex}

\begin{prop}\label{prop:BO}
Let $\phi$ be a compactly supported symplectomorphism of $(V, \beta_{V})$. Writing $(W, \beta_{W})$ for $(V \times \disk^{\ast}\Circle, \beta_{V} + p\rd q)$ with there is a contactomorphism $\psi_{\BO}$ of $\partial W$ for which
\begin{equation*}
    \BO(V,\phi)=\DG(W, W,\psi_{\BO}).
\end{equation*}
\end{prop}
\begin{proof}
As in \Cref{ex:trivial_BO}, we have $f=\alpha_{\mathrm{BO}}(\partial_y)=\Phi_2$. Now notice that in \eqref{eqn:contact_OB}, $\lambda+b(\theta)+ K \rd \theta$ on $\Sigma_{\phi}$ for non-decreasing $g(t)$ such that $b(\theta)=0$ when $\theta \in [0,\pi]$ and $b=1$ near $2\pi$ defines a contact structure isotopic to \eqref{eqn:contact_OB}. The construction in \Cref{thm:bourgeois} works for such contact forms on open books as well, yielding isotopic $\Circle$-invariant contact structures. Such modification does not change the function $f$ and dividing set $\Gamma$, and after such change, the $W_+$ piece of such modification is identified with the $W_+$ piece of $\BO(V,\Id)$. In other words, up to homotopy of ideal Liouville domains, the $W_+$ piece is $V \times \disk^*\Circle$. Similarly, we can arrange $b(\theta)=1$ for $\theta\in [\pi,2\pi]$ and $b=0$ near $0$ to argue that the $W_-$ piece up to homotopy is also $V \times \disk^*\Circle$.

To explain the gluing map, we use $b(\theta)$ such that $b(\theta)=0$ when $\theta\in [0,\pi]$ and $b=1$ near $1$. Then the dividing set looking from the $W_+$ side is
\begin{equation*}
    \partial(V \times \disk^*\Circle_x)=\left(\left(V \times \{(-1,0)_{a,b}\} \cup \Sigma \times \{(1,0)_{a,b}\}\cup \partial V \times  [-1,1]_a \times \{0\}_b \right)\times \Circle_x, \lambda+a\rd x)\right).
\end{equation*}
Now we have $\psi: (\partial(V \times \disk^*\Circle), \beta_{V}+p\rd q) \to  (\partial(V \times \disk^*\Circle), \alpha)$ defined by $\phi^{-1}$ on $V \times \{(1,0)\}\times \Circle $ and $\Id$ elsewhere. Here $\alpha$ is $\phi^*\beta_{V}+p\rd q=\beta_{V}+\eta+p\rd q$ on $V \times \{(1,0)\}\times \Circle $ and $\lambda+p\rd q$ elsewhere. Therefore $\BO(V,\phi)=\DG(V \times \disk^*\Circle,W_-,\psi)$, where $W_-$ is an ideal Liouville filling of the ``non-standard" boundary $(\partial(V \times \disk^*\Circle), \alpha)$ homotopic to $V \times \disk^*\Circle$. To make $V_-$ standard,  $\psi_{\BO}$ is the composition of $\psi$ with the contact isotopy induced from tuning down the $\eta$-component on $V \times \{(1,0)\}\times \Circle$.
\end{proof}

\section{Holomorphic foliations on $\completion{W} = \completion{V} \times \completion{\Sigma}$}\label{Sec:Foliation}

Here through \S \ref{Sec:VanishingAug} we study augmentations associated to Liouville manifolds whose completions are of the form
\begin{equation*}
\left(\completion{W}, \beta_{W}\right) = \left(\completion{V} \times \completion{\Sigma}, \beta_{V} + \beta_{\Sigma}\right)
\end{equation*}
where $(V, \beta_{V})$ is an arbitrary Liouville domain and $(\Sigma, \beta_{\Sigma})$ is a non-simply connected Liouville domain of $\dim \Sigma = 2$. In this section we endow such $\completion{W}$ with codimension-$2$ holomorphic foliations having special properties described in Lemma \ref{Lemma:FoliationProperties}, below. To start we set up some notation.

\subsection{Prerequisites}

The following classes of almost complex structures are used to define symplectic field theory invariants of contact and Liouville manifolds in \cite{BH-cylindrical,BH}.

\begin{defn}\label{def:tame}
Let $\alpha$ be a contact form for a contact manifold $(Y^{2n+1}, \xi)$ with Reeb field $R$ and let $s$ be a coordinate on $I\subset \R$.
An almost complex structure $J$ on an $I$-symplectization $I\times Y$ is \emph{$\alpha$-tame} if
\begin{enumerate}
    \item $J$ is invariant under translation in the $s$-coordinate,
    \item $J(\partial_s)=F_J R$ for some $F_J\in \Cinfty(Y,(0,\infty))$,
    \item there is a $2n$-plane field $\xi_J \subset TY$ satisfying $J\xi_J=\xi_J$, and
    \item $\rd \alpha(V, JV)>0$ for all non-zero $V\in \xi_J$.
\end{enumerate}
\end{defn}

In general, $\alpha$ will be nontrivial on $\xi_J$ so that this hyperplane field will not necessarily agree with $\xi$. Nevertheless, $J$ will be tamed by $\rd (e^{\delta s }\alpha)$ for $\delta \ll 1$, see \Cref{lemma:tame}. This notion of almost complex structures will simplifies the asymptotic analysis for punctured holomorphic curves asymptotic to Reeb orbits as well as the construction of holomorphic foliations. The relevant SFT compactness theorem for $\alpha$-tame almost structures can be found in \cite[\S 3.4]{BH-cylindrical}.

\begin{lemma}\label{lemma:tame}
    Let $\alpha$ be a contact form, we fix a  $2n$-plane field $\xi_J\subset TY$ that is transverse to the Reeb vector field $R$. We fix an almost complex structure $J$ on $\xi_J$ such that $\rd \alpha(V, JV)>0$ for all non-zero $V\in \xi_J$. Then there exists $C>0$, as long as $J(\partial_s)=F_JR$ for $F_J\ge C$ on $I\times Y$, we have $J$ is tamed by with $\rd(e^s\alpha)$.
\end{lemma}
\begin{proof}
    For $U=a\partial_s+bR$ and $V\in \xi_J$, it suffices to prove
    $$\rd(e^s\alpha)(U+V,JU+JV)=a^2F_J+\frac{b^2}{F_J}+a\alpha(JV)+\rd\alpha(V,JV)>0$$
    if $(U,V)\ne 0$. Since $\alpha(JV)^2 \le M \rd \alpha(V,JV)$ for some $M>0$,  if $F_J>M/2$, we have $\rd(e^s\alpha)(U+V,JU+JV)>0$ if $(U,V)\ne 0$.
\end{proof}

Let $(V,\beta_V)$ be a Liouville domain and $\widehat{V}=V\cup (1,\infty)_\sigma \times \partial V$, where $\sigma$ is coordinate on $\widehat{V}$ outside the skeleton of $V$, such that $\sigma \partial_{\sigma}$ is the Liouville vector field and $\log(\sigma)$ is the cylindrical coordinate used in \Cref{def:tame}. We will refer to $\sigma$ as the exponential cylindrical coordinate. We introduce the following notation: For $I$ an interval contained in $\R_{\ge 0}$, which could be closed, open, or open at one end and closed at the other, 
$$V^I:=I_{\sigma}\times \partial V \subset \completion{V}, \text{ if }0\notin I, \quad V^I:=\widehat{V}\backslash \left((\R_+\backslash I)_{\sigma}\times \partial V\right) \subset \completion{V}, \text{ if }0\in I.$$

Now let $(V, \beta_{V})$ and $(\Sigma, \beta_{\Sigma})$ be compact Liouville domains with $\dim \Sigma = 2$. We use $\sigma,p$ as the exponential cylindrical coordinates on $\completion{V},\completion{\Sigma}$ respectively. Here and throughout, 
\begin{enumerate}
    \item $\epsilon$ is a small positive constant;
    \item $Y$ is contact boundary of $V$, with contact form $\alpha_Y=\beta_V|_{Y}$, contact structure $\xi_Y=\ker \alpha_Y$ and Reeb vector field $R_Y$;
    \item $\Circle_q=[0,D]_q/0\sim D$ is a component of the contact boundary of $\Sigma$, with contact form $\rd q$ and Reeb vector field $\partial_q$. Hence $\beta_\Sigma$ in the collar of $\Circle_q$ is $p\rd q$ for $p\in (0,\infty)$.
\end{enumerate}
Liouville vector fields for $\beta_{V}$ and $\beta_{\Sigma}$ will be written as $X_{V}$ and $X_{\Sigma}$, respectively, which outside the skeleton, i.e.\ where the exponential cylindrical coordinate is defined,  take the form
\begin{equation*}
X_{V} = \sigma\partial_{\sigma}, \quad X_{\Sigma} = p\partial_{p}.
\end{equation*}

We write $W^{\square} = V \times \Sigma$ which we equip with the Liouville form $\beta_{W} = \beta_{V} + \beta_{\Sigma}$ and Liouville vector field $X_{W} = X_{V} + X_{\Sigma}$. For a function $f$ on $U = V, \Sigma$, or $W$, the associated Hamiltonian vector field with respect to $d\beta_{U}$ will be written $\fX^{U}_{f}$ and is defined by the convention that $df = d\beta_{U}(\ast, \fX^{U}_{f})$. Define the completion $\completion{W} = \completion{V} \times \completion{\Sigma}$ which we equip with the Liouville form $\beta_{W}$ and which contains $W^{\square}$.

The goal of this section is to prove the following lemma.

\begin{lemma}\label{Lemma:FoliationProperties}
There is a contact hypersurface $\divSet \subset W^{\square}$, such that $X_{W}\pitchfork \divSet$, and a $d\beta_{W}$-tame\footnote{In the usual sense, i.e.\ $\rd \beta_W(U,J_WU)>0$ if $U\ne 0$.} almost complex structure $J_{W}$ on $\completion{W}$ such that $J_{W}$ is $\alpha_{\divSet}=\beta_{W}|_{\divSet}$-tame in the sense of \Cref{def:tame} on a half-cylindrical end $[1, \infty)_{s} \times \divSet$ of $\completion{W}$
and $(\completion{W}, J_{W})$ has a holomorphic foliation $\foliation$ with the following properties:
\be
\item The leaves are parameterized by $\completion{\Sigma}$ by a diffeomorphism
$\completion{V}\times \completion{\Sigma}\to \completion{W}, (v,z)\mapsto \leafInclusion_z(v)$, when $v\in V^{\le \epsilon-\delta}$ for a $0<\delta \ll \epsilon$, we have the projection of $\leafInclusion_z(v)$ to $\completion{\Sigma}$ is $z$.
\item Along $(\epsilon, \infty)_{\sigma} \times Y \times \Sigma \cup \{(\sigma,p)|\sigma \ge \epsilon p, p\ge 1\}\times Y \times \partial \Sigma $, the tangent space of the leaves agrees with the distribution $\langle J_{W}R_{Y}, R_{Y}\rangle \oplus \xi_{Y}$. 
\ee
\end{lemma}
In proving the lemma, some specific properties of the Reeb vector field $R_{\divSet}$ will be described. We will also later use the specific description of $J_{W}$ to control the asymptotics of the $\leaf_{z}:=\im \leafInclusion_z$. Both will be important for our contact homology computations. Using the induced $\alpha_{\divSet}$ tame almost complex structure $J_{\divSet}$ on the symplectization $\R \times \divSet$, we will see in \Cref{lemma:foliation_symp} that this symplectization also admits a holomorphic foliation, which can be viewed as the positive asymptotic of the foliation on $\completion{W}$.

\subsection{The contact hypersurface $\divSet$ and its symplectization}

We define a contact hypersurface $\divSet$ in $W^{\square}$ as the union of the graphs of two functions $f_{V}$ and $f_{\Sigma}$. We will then define $W$ to be the subset of $W^{\square}$ bound by $\divSet$ and see the symplectization of $\divSet$ as a collar of $\completion{W}$. Along the way, we will describe Reeb dynamics on $\divSet$, equipped with the contact form it inherits from $\beta_{W}$ on $W^{\square}$.

\subsubsection{The hypersurface $\divSet_{V}$}

Let $f_{V}: V^{\le 1-2\epsilon} \rightarrow [2\epsilon, 1]$ be such that
\be
\item $f_{V} = 1$ on $V^{\le \epsilon}$, and
\item in $V^{[\epsilon,1-2\epsilon]}$, $f_{V} = f_{V}(\sigma)$ with $\frac{\partial f_{V}}{\partial \sigma} \leq 0$ and $f_{V}(\sigma) = 1-\sigma$ along $\{\sigma \in [2\epsilon, 1-2\epsilon]\}$.
\ee
Using $f_{V}$ we define a map $\Phi^{V}$ by
\begin{equation*}
\begin{gathered}
\Phi^{V}: [1, \infty)_{s} \times V^{\le 1-2\epsilon}_{v} \times \Circle_{q} \rightarrow \completion{V} \times [2\epsilon, \infty)_{p} \times \Circle_{q} \subset \completion{W}\\
\Phi^{V}(s, v, q) = \Flow^{\log(s)}_{X_{W}}\left(v, f_{V}, q\right) = \left(\Flow^{\log(s)}_{X_{V}}v, sf_{V}, q\right),\\
\implies (\Phi^{V})^{\ast}\beta_{W} = s\left(f_{V}\rd q + \beta_{V}\right).
\end{gathered}
\end{equation*}
We define a hypersurface $\divSet_{V} \subset W$ by $\Phi^{V}(\{s=1\})$. It follows that the Reeb vector field $R_{\divSet}$ on $\divSet_{V}$ takes the form
\begin{equation*}
R_{\divSet} = \left(f_{V} - \beta_{V}(\fX_{f_{V}}^{V})\right)^{-1}\left( \partial_{q} - \fX^{V}_{f_{V}}\right).
\end{equation*}

\subsubsection{The hypersurface $\divSet_{\Sigma}$}\label{Sec:divSetSigmaDef}
Let $f_{\Sigma}: \Sigma^{\le 1-2\epsilon} \rightarrow [2\epsilon, 1]$ be such that
\be
\item $f_{\Sigma}$ is $\Cinfty$ close to $1$ on $\Sigma^{\le \epsilon}$, and
\item on $\Sigma^{[\epsilon,1-2\epsilon]}$, $f_{\Sigma} = f_{\Sigma}(p)$ with $\frac{\partial f_{\Sigma}}{\partial p} < 0$ and $f_{\Sigma}(p) = 1-p$ along $\{ p\in [2\epsilon, 1-2\epsilon]\}$.
\ee
Using $f_{\Sigma}$ we define a map $\Phi^{\Sigma}$ and hypersurface $\divSet_{\Sigma}$ by
\begin{equation*}
\begin{gathered}
\Phi^{\Sigma}:[1, \infty)_{s} \times Y_{y} \times \Sigma^{\le 1-2\epsilon}_{z} \rightarrow [2\epsilon, \infty)_{\sigma} \times Y \times \completion{\Sigma} \subset \completion{W},\\
\Phi^{\Sigma}(s, y, z) = \Flow^{\log(s)}_{X_{W}}\left(f_{\Sigma}, y, z\right) = \left(sf_{\Sigma}, y, \Flow^{\log(s)}_{X_{\Sigma}}z\right),\\
\implies (\Phi^{\Sigma})^{\ast}\beta_{W} = s\left(f_{\Sigma}\alpha_{Y} + \beta_{\Sigma}\right).
\end{gathered}
\end{equation*}
Similarly, we define  $\divSet_{\Sigma} = \Phi^{\Sigma}(\{s=1\})$. It follows that the Reeb vector field on $\divSet_{\Sigma}$ takes the form
\begin{equation}\label{Eq:ReebMSigma}
R_{\divSet} = \left(f_{\Sigma} - \beta_{\Sigma}(\fX_{f_{\Sigma}}^{\Sigma})\right)^{-1}\left( R_{Y} - \fX^{\Sigma}_{f_{\Sigma}}\right)
\end{equation}
Along $\Phi^{\Sigma}\left(\{s=1\}\times Y\times \Sigma^{[\epsilon,1-2\epsilon]}\right)$, this specializes to
\begin{equation*}
\fX_{f_{V}}^{V} = \frac{\partial f_{\Sigma}}{\partial p}\partial_{q}, \quad R_{\divSet} = \left(f_{\Sigma} - p\frac{\partial f_{\Sigma}}{\partial p}\right)^{-1}\left(R_{Y} - \frac{\partial f_{\Sigma}}{\partial p}\partial_{q}\right)
\end{equation*}
Restricting further to the subset $\Phi^{\Sigma}\left(\{s=1\}\times Y\times \Sigma^{[2\epsilon,1-2\epsilon]}\right)$ we get
\begin{equation*}
X_{f_{V}}^{V} = -\partial_{q}, \quad R_{\divSet} = R_{Y} + \partial_{q}
\end{equation*}

\subsubsection{The overlap}

The $\divSet_{V}$ and $\divSet_{\Sigma}$ overlap along the set
\begin{equation*}
\divSet_{o} = \{ p = 1-\sigma :\  \sigma \in (2\epsilon, 1-2\epsilon)\} \subset V^{\ge 2\epsilon} \times \Sigma^{\ge 2\epsilon}
\end{equation*}
along which
\begin{equation*}
 \quad R_{\divSet} = \partial_{q} + R_{Y}
\end{equation*}

We can therefore define $\divSet = \divSet_{V} \cup \divSet_{\Sigma}\subset W^{\square}$ to be the union of these pieces equipped with the contact forms $\alpha_{\divSet} = \beta_{W}|_{\divSet}$. Then define
\begin{equation*}
\Phi: (1, \infty)_{s} \times \divSet \rightarrow W \quad  \implies \quad \Phi^{\ast}\beta_{W} = s\alpha_{\divSet}
\end{equation*}
as the ``union'' of the $\Phi^{V}$ and $\Phi^{\Sigma}$, i.e.\ $\Flow^{\log(s)}_{X_{W}}$ from $\Gamma$. We can set $W \subset W^{\square}$ to be the complement of image of $(0, \infty)\times \divSet$ under $\Phi$. So the image of $\Phi$ gives a half-cylindrical end of the completion of $W$ in $\completion{W}$.

The following lemma gives us a quick summary of the dynamics of $R_{\divSet}$ and follows immediately from the fact that $R_{\divSet}$ has positive $\partial_{q}$ coefficient outside $\Gamma_{\Sigma^{\le \epsilon}}:=\Phi^V(\{s=1\} \times Y \times \Sigma^{\le \epsilon})$. More analysis will be provided in \S \ref{Sec:DetailedDynamics}.

\begin{lemma}\label{Lemma:ContractibleOrbitLocation}
When $\Sigma$ is not a disk, for any choices of $f_{V}, f_{\Sigma}$ as above, the contact hypersurface $\divSet$ is such that all contractible orbits of $R_{\divSet}$ are contained in the subset $\Gamma_{\Sigma^{\le \epsilon}}$.
\end{lemma}

\subsection{Definition of $J_{W}$ along $\im\Phi$}

Now we define an $\alpha_{\divSet}$-tame almost complex structure $J_{W}$ on the half-cylindrical end $\im \Phi$ of $\completion{W}$. We specify $J_{W}\partial_{s}$ as
\begin{equation}\label{eqn:J_liouville}
J_{W}s\partial_{s} =C \cdot \begin{cases}  \partial_{q} - \fX^{V}_{f_{V}} & \text{along}\ \im\Phi^{V}\\
R_{Y} - \fX^{\Sigma}_{f_{\Sigma}} & \text{along}\ \im\Phi^{\Sigma}\\
R_{Y} + \partial_{q} & \text{along the overlap}.
\end{cases}
\end{equation}
In other words, $J_Ws\partial_s=CfR_{\divSet}$ for $f:\divSet \to \R_{>0}$, where $f=f_V-\beta_V(\fX^V_{f_V})$ on $\divSet_V$, $f=f_\Sigma-\beta_\Sigma(\fX^\Sigma_{f_\Sigma})$ on $\divSet_{\Sigma}$ and $f=1$ on the overlap $\Gamma_o$. Here $C>0$ is a constant, which will be specified later to make sure \Cref{lemma:tame} applies.

\subsubsection{$J_{W}$ along $\im \Phi^{V}\backslash \im \Phi^{\Sigma}$} Note that $\im \Phi^{V}\backslash \im \Phi^{\Sigma}$ is $\Phi^{V}([1,\infty)_s\times V^{\le 2\epsilon}\times \Circle_q)$. Along $\Phi^{V}([1, \infty)_{s} \times V^{\le \epsilon} \times \Circle_q)$ the above expression yields $J_{W}s\partial_{s} = \partial_{q}$. It only remains to determine $J_{W}|_{TV}$ and we declare that
\begin{equation*}
J_{W}|_{TV} = J_{V}
\end{equation*}
where $J_{V}$ is a $d\completion{\beta}_{V}$-compatible almost complex structure on $\completion{V}$, which preserves the contact structure $\xi_Y$ and sends $\sigma \partial_{\sigma}$ to $R_Y$ near the boundary of $V^{\le \epsilon}$ and $\R_+$-invariant (w.r.t.\ the dilation action as we are using the exponential cylindrical coordinate) on $\completion{V}\backslash V^{\le \epsilon}$. We write $J_{\xi_Y}$ for the restriction of $J_V$ to $\xi_Y$ near the boundary.

Along $\Phi^V([1,\infty)_s\times V^{[\epsilon,2\epsilon]}\times \Circle_q)$, the hypersurface $\divSet_{V}$ is the graph of $p=f_{V}(\sigma)$ for $\sigma\in [\epsilon,2\epsilon]$. There we compute $J_{W}$ on $\Phi^V([1,\infty)_s\times V^{[\epsilon,2\epsilon]}\times \Circle_q)$ as
\begin{equation}\label{Eq:JMV}
\begin{gathered}
X_{W} = s\partial_{s} = \sigma \partial_{\sigma} + p\partial_{p}, \quad T\divSet_{V} =\langle R_Y \rangle \oplus  \xi_Y \oplus \langle \partial_q, \partial_{\sigma}+\frac{\rd f_V}{\rd \sigma} \partial_p \rangle\\
J_{W}|_{\xi_Y}=J_{\xi_Y},\quad J_W(\sigma \partial_{\sigma}+\sigma \frac{\rd f_V}{\rd \sigma} \partial_p ) = R_Y, \quad  J_{W}X_{W} = C\left(\partial_{q}-\fX^V_{f_V}\right)=C\left(\partial_{q}-\frac{\rd f_V}{\rd \sigma} R_Y\right),\\
\implies J_{W}\partial_{p} = \left(f_V-\sigma \frac{\rd f_V}{\rd \sigma}\right)^{-1}\left(C\partial_{q} - \left(1+C\frac{\rd f_V}{\rd \sigma}\right)R_Y\right).
\end{gathered}
\end{equation}
The discussion can be extended to $2\epsilon \le \sigma \le 1-2\epsilon$, which simplifies to
\begin{equation}\label{eqn:J_sigma_p_V}
    J_{W}\partial_{p} = C\partial_{q}+(C-1)R_Y, \quad J_{W}\partial_{\sigma} = R_Y/\sigma+ C\partial_q+(C-1)R_Y, \quad  J_W(\partial_{\sigma}-\partial_p)=R_Y/\sigma.
\end{equation}
However we will need to change $J_W$ in the region $2\epsilon \le \sigma \le 1-2\epsilon$, spelling out the formula will help us to see the smoothness of $J_W$ of the final construction.

We define $\xi_{J_W}$ to be $TV$ for $\sigma \le \epsilon$ and  $\xi_Y\oplus \langle \partial_{\sigma}+\frac{\rd f_V}{\rd \sigma} \partial_p, R_Y \rangle$ when $\epsilon \le \sigma \le 2\epsilon$. This decomposition is compatible with the complex structure. Again the definition extends to $2\epsilon \le \sigma \le 1-2\epsilon$, however, the actual definition of $\xi_{J_W}$ when  $2\epsilon \le \sigma \le 1-2\epsilon$ will be modified later. We write $\eta =  \langle \partial_{\sigma}+\frac{\rd f_V}{\rd \sigma} \partial_p, R_Y \rangle$ for $\epsilon \le \sigma \le 2\epsilon$, which also extends to $2\epsilon \le \sigma \le 1-2\epsilon$.

\begin{prop}\label{prop:compatible_V}
    $J_W$ preserves $\xi_{J_W}$ on $\divSet_V\backslash \divSet_{\Sigma}$. $\rd\alpha_{\divSet},J_W$ are compatible on $\xi_{J_W}$ over $\divSet_V\backslash \divSet_{\Sigma}$. In particular, $J_W$ is $\alpha_{\divSet}$-tame on $\divSet_V\backslash \divSet_{\Sigma}$.
\end{prop}
\begin{proof}
It is clear from the construction that $J_W$ preserves $\xi_{J_W}$. Since $\alpha_\divSet=\beta_V+\rd q$ and $\xi_{J_W}=TV$ when $\sigma \le \epsilon$, $J_W$ is compatible with $\rd \alpha_{\divSet}$ on $\xi_{J_W}$ as $J_V$ is compatible with $\rd \beta_V$ on $TV$.  

When $\epsilon \le \sigma \le 2\epsilon$, we have $\alpha_\divSet = \sigma \alpha_Y +p\rd q$ restricted to $p=f_{V}(\sigma)$. Therefore 
$$\rd \alpha_{\divSet} = \rd \sigma \wedge \alpha_Y +\sigma \rd \alpha_Y + \rd p\wedge \rd q.$$
Since $J_W$ preserves $\xi_Y$, and $\rd\alpha_{\divSet}|_{\xi_Y}=\sigma \rd \alpha_Y$, we have $J_W$ is compatible with $\rd \alpha_{\divSet}$ on $\xi_Y$ as $J_{\xi_Y}$ is compatible with $\rd \alpha_Y$ on $\xi_Y$. Then we can compute 
$$\rd \alpha_{\divSet} \left(\partial_{\sigma}+\frac{\rd f_V}{\rd \sigma}\partial_p, J_W\left(\partial_{\sigma}+\frac{\rd f_V}{\rd \sigma}\partial_p\right)\right)=1/\sigma>0$$
Hence $J_W$ is compatible with $\rd \alpha_{\divSet}$ on $\eta$. Finally, for $v\in \xi_Y$ and $u\in \eta$, we have $\rd \alpha_{\divSet}(v,u)=0$. Hence $\rd\alpha_{\divSet},J_W$ are compatible on $\xi_{J_W}$ over $\divSet_V\backslash \divSet_{\Sigma}$.
\end{proof}

\subsubsection{$J_{W}$ along $\im \Phi^{\Sigma}\backslash \im \Phi^{V}$}
Note that $\im \Phi^{\Sigma}\backslash \im \Phi^{V}$ is $\Phi^{\Sigma}([1,\infty)_s\times Y \times \Sigma^{\le 2\epsilon})$. So to complete the definition of $J_{W}$, note that we can identify $\divSet_{\Sigma}\cap (\im \Phi^{\Sigma}\backslash \im \Phi^{V})$ with the graph of $\sigma=f_{\Sigma}$ and
$$X_{W} = s\partial_{s} = \sigma\partial_{\sigma} + X_{\Sigma}, \quad  T\divSet_{\Sigma} = \langle R_Y \rangle \oplus \xi_Y \oplus \{ \zeta + \rd f_{\Sigma}(\zeta)\partial_{\sigma}|\zeta \in T\Sigma\},$$
we fix a $d\beta_{\Sigma}$-compatible almost complex structure $J_{\Sigma}$ for which $J_{\Sigma}p\partial_{p} = \partial_{q}$ near the boundary of $\Sigma^{\le \epsilon}$ and $\R_+$-invariant on $\completion{\Sigma}\backslash \Sigma^{\le 2\epsilon}$. 
We define 
\begin{equation}\label{Eq:JAlongMSigma}
\begin{gathered}
J_{W}|_{\xi_{Y}} = J_{\xi_Y},\\
J_{W}\left(\zeta + \rd f_{\Sigma}(\zeta)\partial_{\sigma}\right) = J_{\Sigma}\zeta + \rd f_{\Sigma}(J_{\Sigma}\zeta)\partial_{\sigma}, \quad \zeta \in T\Sigma\\
J_{W}X_{W} = C(R_{Y} - \fX^{\Sigma}_{f_{\Sigma}})
\end{gathered}
\end{equation}
Along the set $\epsilon \le p \le  2\epsilon$ where $f_{\Sigma} = f_{\Sigma}(p)$ and $\fX^{\Sigma}_{f_\Sigma}=\frac{\rd f_{\Sigma}}{\rd p}\partial_q$, this simplifies to
$$J_W \partial_\sigma = \left(f_{\Sigma}-p\frac{\rd f_{\Sigma}}{\rd p}\right)^{-1}\left(CR_Y-\left(1+C\frac{\rd f_{\Sigma}}{\rd p}\right)\partial_q\right)$$
The discussion can be extended to $2\epsilon \le p \le 1-2\epsilon$, which simplifies to
\begin{equation}\label{eqn:J_sigma_p_Sigma}
 J_W(\partial_{\sigma}-\partial_p)=-\partial_q/p.
\end{equation}
We define $\xi_{J_W}$ to be  $\xi_Y \oplus \{ \zeta + \rd f_{\Sigma}(\zeta)\partial_{\sigma}|\zeta \in T\Sigma\}$ and write $\eta=\{ \zeta + \rd f_{\Sigma}(\zeta)\partial_{\sigma}|\zeta \in T\Sigma\}$ on $\divSet_\Sigma\backslash \divSet_{V}$, which is $\langle \partial_q, \partial_p+\frac{\rd f_\Sigma}{\rd p}\partial_\sigma  \rangle $ on $\epsilon\le p \le 2\epsilon$. The decomposition $\xi_{J_W}$ to be  $\xi_Y \oplus \eta$ is compatible with the complex structure. 

\begin{prop}
    $J_W$ preserves $\xi_{J_W}$ on $\divSet_\Sigma\backslash \divSet_{V}$, and $\rd\alpha_{\divSet},J_W$ are compatible on $\xi_{J_W}$ over $\divSet_{\Sigma}\backslash \divSet_{V}$. In particular, $J_W$ is $\alpha_\divSet$-tame on $\divSet_\Sigma\backslash \divSet_{V}$.
\end{prop}
\begin{proof}
It is clear from the construction that $J_W$ preserves $\xi_{J_W}$. Note that we have $\alpha_\divSet = \sigma \alpha_Y +\beta_\Sigma$ restricted to $\sigma=f_{\Sigma}$. Therefore 
$$\rd \alpha_{\divSet} = \rd \sigma \wedge \alpha_Y +\sigma \rd \alpha_Y + \rd \beta_\Sigma.$$
Since $J_W$ preserves $\xi_Y$,  we have $J_W$ is compatible with $\rd \alpha_{\divSet}$ on $\xi_Y$. Then for $\zeta\ne 0 \in T\Sigma$, we can compute 
$$\rd \alpha_{\divSet} \left(\zeta + \rd f_{\Sigma}(\zeta)\partial_{\sigma}, J_W\left(\zeta + \rd f_{\Sigma}(\zeta)\partial_{\sigma}\right)\right)=\rd \beta_\Sigma(\zeta,J_{\Sigma}\zeta)>0.$$
Hence $J_W$ is compatible with $\rd \alpha_{\divSet}$ on $\eta$. Finally, for $v\in \xi_Y$ and $u\in \eta$, we have $\rd \alpha_{\divSet}(v,u)=0$. Hence $\rd\alpha_{\divSet},J_W$ are compatible on $\xi_{J_W}$ over $\divSet_{\Sigma}\backslash \divSet_{V}$.
\end{proof}

\subsubsection{$J_{W}$ along the overlap $\divSet_o=\divSet_V\cap \divSet_{\Sigma}$}
To complete the definition of $J_{W}$, it suffices to do so along the overlapping region $\divSet_{o} = \{ p=1-\sigma: \sigma \in (2\epsilon, 1-2\epsilon)\}$. Here the contact structure is given by 
$$\ker(\alpha_{\divSet}) = \xi_{Y} \oplus \langle \partial_{\sigma}-\partial_p, pR_Y-(1-p)\partial_q\rangle.$$
Recall that, with the previous constructions, we have for $\delta \ll 1$, 
$$\eta = \langle \partial_\sigma - \partial_p, R_Y\rangle, \quad \sigma \in [2\epsilon,2\epsilon+\delta);$$
$$\eta = \langle \partial_\sigma - \partial_p, \partial_q\rangle, \quad p \in [2\epsilon,2\epsilon+\delta).$$
Now we choose a smooth function $b:[0,1]_{\sigma}\to [0,1]$, such that 
\begin{enumerate}
    \item $b=0$ for $\sigma\le 2\epsilon+\delta$, and $b=1$ for $\sigma \ge 1-2\epsilon-\delta$;
    \item $\frac{\rd b}{\rd \sigma }>0$ on $(2\epsilon+\delta,1-2\epsilon-\delta)$;
    \item $b(1-\sigma)=1-b(\sigma)$.
\end{enumerate}
Then we define $\eta$  on the overlap $\Gamma_o$ by
$$\langle \partial_\sigma -\partial_p, (1-b(\sigma))R_Y-b(\sigma)\partial_q \rangle,$$
and $\xi_{J_W}:=\xi_Y\oplus \eta$.
Therefore $\eta$ can be patched together smoothly, so does $\xi_{J_W}$.
To define $J_W$ on $\Gamma_0$, we require 
$$ J_{W}s\partial_{s}= R_Y+\partial_q, \quad J_W|_{\xi_Y} = J_{\xi_Y}, J(\partial_\sigma -\partial_p)=(1-b(\sigma))/\sigma \cdot R_Y-b(\sigma)/p \cdot \partial_q.$$
Then by \eqref{eqn:J_sigma_p_V} and \eqref{eqn:J_sigma_p_Sigma}, $J_W$ is smoothly defined on $\im \Phi$. 
\begin{prop}
    On $\divSet_o$, we have $J_W$ preserves $\xi_{J_W}$ and $\rd\alpha_{\divSet},J_W$ are compatible on $\xi_{J_W}$. In particular, $J_W$ is $\alpha_{\divSet}$-tame on $\divSet_o$.
\end{prop}
\begin{proof}
On $\divSet_o$, we have $\alpha_\divSet = \sigma \alpha_Y +p\rd q$ restricted to $p=1-\sigma$. Therefore 
$$\rd \alpha_{\divSet} = \rd \sigma \wedge \alpha_Y +\sigma \rd \alpha_Y + \rd p\wedge \rd q.$$
Since $J_W$ preserves $\xi_Y$, it is clear that $J_W$ is compatible with $\rd \alpha_{\divSet}$ on $\xi_Y$. Then we can compute 
$$\rd \alpha_{\divSet} \left(\partial_{\sigma}-\partial_p, J_W\left(\partial_{\sigma}-\partial_p\right)\right)=(1-b(\sigma))/\sigma+b(\sigma)/p>0$$
Hence $J_W$ is compatible with $\rd \alpha_{\divSet}$ on $\eta$. Finally, for $v\in \xi_Y$ and $u\in \eta$, we have $\rd \alpha_{\divSet}(v,u)=0$. Hence $\rd\alpha_{\divSet},J_W$ are compatible on $\xi_{J_W}$ over $\divSet_o$.
\end{proof}

So far, we have built $J_W$ on the positive cylindrical end $\im \Phi$ that is invariant under the $\R_+$ action, and $J_W|_{\xi_{J_W}}$ is independent of $C$.  It is clear that we can extend it to an almost complex structure on $\completion{\divSet}$ with the following properties.
\begin{prop}
For any $C>0$, we define $J_{\divSet}$ to be the almost complex structure on $\completion{\divSet}$ by extending $J_W$ above from the positive cylindrical end to the whole symplectization $\completion{\divSet}$ that is $\R_+$-invariant, then $J_{\divSet}$ is $\alpha_\divSet$-tame in the sense of \Cref{def:tame}, with $\xi_{J_{\divSet}}=\xi_{J_W}$.
\end{prop}

As a warm-up, we can establish the following holomorphic foliation on $\completion{\divSet}$, a special case when $\Sigma = [-1,1]_p\times \R_q$ was constructed in \cite[\S 5]{Avdek:Hypersurface}.
\begin{lemma}\label{lemma:foliation_symp}
    $(\completion{\divSet},J_{\divSet})$ has a holomorphic foliation with the following leaves:
    \begin{enumerate}
        \item $\completion{V}$-leaves, if we project the part in $(0,\infty)_s\times (\divSet\backslash \divSet_{V})$ to $(\divSet\backslash \divSet_{V})$ then to $\Sigma$ induced from the projection from $\completion{V}\times \completion{\Sigma}\to \completion{\Sigma}$ \footnote{The composition is not the projection $\completion{\divSet}\subset \completion{V}\times \completion{\Sigma}\to \completion{\Sigma}$!}, we get gradient flows of $f_{\Sigma}$ converging to critical points of $f_{\Sigma}$, such a leaf is biholomorphic to $(\completion{V},J_V)$.
        \item $\completion{Y}$-leaves, which are completely contained in  $(0,\infty)_s \times \divSet_{\Sigma}$ and whose image under the projection to $\divSet_{\Sigma}$ then to $\Sigma$ are gradient flows of $f$ with both ends converging to critical points of $f_{\Sigma}$ (could rest on a critical point), such a leaf is biholomorphic to  $(\completion{Y}, J_Y)$, where $J_Y$ is the induced complex structure on $Y=\partial V$ from $J_V$.
    \end{enumerate}
    Such a foliation is denoted by $\foliation_{\divSet}$.  The $s$-translation sending levels of $\foliation_{\divSet}$ to leaves, preserving the leave type.
\end{lemma}
\begin{proof}
    Over $(0,\infty)_s\times \divSet_{V^{\le \epsilon}}$, $\Phi^V(\{s\} \times V^{\le \epsilon} \times \{q\})$ is clearly $J_W$ holomorphic. Over $\Phi^V((0,\infty)_s \times V^{[\epsilon,2\epsilon]} \times \{q\})$, we can compute that
    $$J_W R_Y=-\sigma \partial_{\sigma}-\sigma \frac{\rd f_V}{\rd \sigma} \partial_p$$
    In particular, 
    $$Y\times I\to \completion{\divSet}, \quad (y,t)\mapsto (y,\gamma(t))$$
    has a $J_W$-holomorphic image, where $\gamma:I\to (0,\infty)_\sigma \times (0,\infty)_p$ solves the differential equation $\gamma'=-J_WR_Y= \sigma \partial_{\sigma}+\sigma \frac{\rd f_V}{\rd \sigma}\partial_p$. It is clear that those leaves patch with  $\Phi^V(\{s\} \times V^{\le \epsilon} \times \{q\})$ as $f_V=1$ on $V^{\le \epsilon}$ smoothly. Since $s\partial_s = \sigma \partial_\sigma+p\partial_p$, the projection by forgetting the $s$-cooridnate leaves $-J_WR_Y$ with $-p\partial_p+\sigma \frac{\rd f_V}{\rd \sigma}\partial_p$. As $-p+\sigma \frac{\rd f_V}{\rd \sigma}<0$, the projection to $\Sigma$ are lines in the $-\partial_p$ direction.

    Finally, over the symplectization of $\divSet_{\Sigma}$, we have 
    $$J_W R_Y=-s/C\partial_s+J\fX^{\Sigma}_{f_{\Sigma}}=-s/C\partial_s-\rd \beta_V (\fX^\Sigma_{f_{\Sigma}},J_{\Sigma}\fX^\Sigma_{f_{\Sigma}})\partial_{\sigma}+J_{\Sigma}\fX^{\Sigma}_{f_{\Sigma}}.$$
    The first projection forgets $\partial_s$ and the next map forgets $\partial_{\sigma}$, therefore the map to $\Sigma$ maps $-J_WR_Y$ to $-J_{\Sigma}\fX^{\Sigma}_{f_{\Sigma}}$, which is the gradient vector of $f_{\Sigma}$. As outside $(0,\infty)_s\times \divSet_{V^{\le \epsilon}}$, the folition is constructed by solving $\gamma'=-J_WR_Y$, the constructions on different regions clearly patch up smoothly. Using $\gamma$ as the Liouville direction on leaves, it is clear that $\completion{V},\completion{\divSet}$ leaves are biholomorphic to $(\completion{V},J_V),(\completion{Y},J_Y)$ respectively.
\end{proof}

\begin{rmk}
We do not need the following in this paper in view of \Cref{Lemma:ContractibleOrbitLocation}, but it will be useful for other purposes.
In the definition of $J_W,J_{\divSet}$, we need the following properties, namely, (1) \eqref{eqn:J_liouville}, (2) $J_W$ on $\xi_{J_W}$ is tamed by $\rd \alpha_{\divSet}$. To have the foliation, we need $-J_WR_Y$ is the gradient vector of $f_{\Sigma}$ through the map to $\Sigma$ in the region outside the trivial foliation by $V^{\le \epsilon}$ as in the proof of \Cref{lemma:foliation_symp}. Note that the taming condition is open, and  $-J_WR_Y$'s image to $\Sigma$ will be gradient like for $f_{\Sigma}$, as long as we do not perturb $\divSet_{\Sigma^{\le \epsilon}}$. As a consequence, we can perturb $\alpha_{\divSet}$ to $\alpha'_{\divSet}$ by perturbing $f_{V}$ on $V^{\le \epsilon}$ to be Morse and apply perturbations to contact forms on $\divSet\backslash (\divSet_{\Sigma^{\le \epsilon}}\cup \divSet_{V^{\le \epsilon}})$ to make every Reeb orbits up to an arbitrarily high threshold to be non-degenerate, as there are typically $S^1$-Morse Bott families of Reeb orbits in $\divSet\backslash (\divSet_{\Sigma^{\le \epsilon}}\cup \divSet_{V^{\le \epsilon}})$, see e.g. \cite[\S 6.2]{Zhou:ADC}. This perturbation changes $\rd \alpha'_{\divSet}$ and $R'_{\divSet}$. We define a new almost complex structure $J'_W$ by changing \eqref{eqn:J_liouville} to the version using the perturbed $R'_{\divSet}$ (with the same $C$ and $f$) and requiring that $J'_W=J_W$ restricted to $\xi_{J_W}$. Therefore $J'_W$ is $\alpha'_{\divSet}$ tame. Since  $-J'_WR'_Y$ will have the same property in \Cref{lemma:foliation_symp}, we get a similar foliation $\foliation_{\divSet}$.

\end{rmk}

\subsection{Extension over $W$}
Now we seek to extend $J_{W}$ -- currently only defined on the half-cylindrical end of $\completion{W}$ -- to a $d\beta_{W}$ compatible almost complex structure on $W$. To do so, we fix $0<\delta \ll \epsilon$, we define a function $b_{\Sigma}:\completion{\Sigma}\to \R_+$ such that
\be
\item $b_{\Sigma}=\epsilon p$ on a neighborhood of $\Sigma^{\ge 1}$;
\item $b_{\Sigma}$ only depends on $p$ on a collar neighborhood of $\partial \Sigma^{\le 1}$ and $\frac{\rd b_{\Sigma}}{\rd p}\ge 0$;
\item $b_{\Sigma}=\epsilon-\delta$ outside the above collar neighborhood in $\Sigma^{\le 1}$
\ee
We also pick a function $h_{\Sigma}$ such that 
\be
\item $h_{\Sigma}=1$ on a neighborhood of $\Sigma^{\ge 1}$;
\item $h_{\Sigma}$ only depends on $p$ on a collar neighborhood of $\partial \Sigma^{\le 1}$ and $\frac{\rd h_{\Sigma}}{\rd p}\ge 0$;
\item $h_{\Sigma}=0$ outside the above collar neighborhood in $\Sigma^{\le 1}$
\ee
Using $b_{\Sigma},h_{\Sigma}$, we extend the definition $J_W$ to $\left\{ V^{\le b_{\Sigma}(z)},z\in \Sigma \right\}$ as follows:
\be
\item $J_W|_{V^{\le b_{\Sigma}(z)}}=(\Flow_{X_V}^ {\log(b_{\Sigma}(z)/\epsilon)})_*J_{V^{\le \epsilon}}$;
\item $J_Wp\partial_p =C\partial_q-h_{\Sigma}J_VX_V$ on the domain where $h_{\Sigma}\ne 0$;
\item $J_W\zeta=J’_\Sigma \zeta$ on the domain where $h_{\Sigma}= 0$. Here $J'_{\Sigma}$ is a $\rd \beta_\Sigma$ compatible almost complex structure on $\Sigma$, such that $J'_{\Sigma}p\partial_p=C\partial_q$ in the neighborhood here $h_{\Sigma}\ne 0$
\ee
To see that it extends $J_W$ smoothly from the positive cylindrical end, note that outside $p\ge 1$, we have 
\be
\item $J_W|_{V^{\le p\epsilon}}=(\Flow_{X_V}^ {\log(p)})_*J_{V^{\le \epsilon}}$;
\item $J_Wp\partial p =C\partial_q-J_VX_V$.
\ee
Hence our definition extends smoothly.
\begin{prop}
    On the domain where $J_W$ is defined so far, $J_W$ is tamed by $\rd \beta_W$ for $C\gg 0$.
\end{prop}
\begin{proof}
    By \Cref{lemma:tame}, for $C\gg 0$, we have $J_W$ is tamed by  $\rd \beta_W$ on the positive cylindrical end. Now over $\left\{ V^{\le b_{\Sigma}(z)},z\in \Sigma \right\}$, when $h_{\Sigma}=0$, $J_W=(\Flow_{X_V}^ {\log((\epsilon-\delta)/\epsilon)})_*J_{V^{\le \epsilon}}\oplus J'_{\Sigma}$, which is clearly tamed by $\rd \beta_W$. Now we consider the region where $h_{\Sigma}\ne 0$, for $U\in TV$ and $V=a\partial_p+b\partial_q$, we compute that $\rd \beta_W (U+V,J_WU+J_WV)$ as 
    $$\rd \beta_V(U,J_WU)+\frac{a^2C}{p}+\frac{b^2p}{C}+\rd \beta_V(U,-\frac{ah_{\Sigma}}{p}J_VX_V-\frac{bh_{\Sigma}}{C}X_V).$$
    For $C\gg 0$, $a^2C/p+\frac{1}{2}\rd \beta_V(U,J_WU)$ will beat $-\frac{ah_{\Sigma}}{p}\rd \beta_V(U,J_VX_V)$. 
    Note that 
    $$\rd \beta_V(U,-\frac{bh_{\Sigma}}{C}X_V)=\frac{bh_{\Sigma}}{C}\beta_V(U)$$
    It is clear that for $C\gg 0$, $b^2p/C+\frac{1}{2}\rd \beta_V(U,J_WU)$ will beat $\frac{bh_{\Sigma}}{C}\beta_V(U)$. As a consequence, when $C\gg 0$,  $J_W$ is tamed by $\rd \beta_W$.
\end{proof}

The subset of $W$ within which $J_{W}$ is not yet defined is then a proper subset of $(\epsilon-\delta, \infty)_{\sigma} \times Y \times \Sigma$, where $R_{Y}$ and $\partial_{\sigma}$ are defined.

\begin{lemma}\label{Lemma:JRY}
Everywhere that $J_{W}$ has been so far defined within $(\epsilon-\delta, \infty)_{\sigma} \times Y \times \Sigma$, we have
\begin{equation*}
J_{W}R_{Y} = A\partial_{\sigma} + Z_{\Sigma}
\end{equation*}
satisfying the properties that
\be
\item $A$ is a strictly negative function,
\item $Z_{\Sigma}$ is a vector field with values in $T\Sigma \subset TW$, and
\item both $A$ and $Z_{\Sigma}$ are independent of $y \in Y$.
\ee
\end{lemma}

\begin{proof}
That $A$ is strictly negative follows from the fact that $J_{W}$ is $d\beta_{W}$ compatible and $0<d\beta_{W}(R_{Y}, JR_{Y}) = -d\sigma(JR_{Y}) = -A$. That $J_{W}R_{Y} - A\partial_{\sigma}$ is contained in $T\Sigma$ along $\divSet_{\Sigma}$ follows from Equation \eqref{Eq:JAlongMSigma}. This property continues to hold in a neighborhood of $\divSet_{\Sigma}$ in $\im \Phi$ using the flow invariance property w.r.t.\ $X_W$, since the flow of $X_{W}$ preserves the splitting of $TW$ as $\langle R_Y\rangle \oplus\langle  \partial_{\sigma}, T\Sigma \rangle \oplus \xi_{Y}$. Along the lower boundary from $\sigma = b_{\Sigma}(z)$, we have $J_{W}R_{Y}= -\sigma \partial_{\sigma}$ by $J_W|_{V^{\le b_{\Sigma}(z)}}=(\Flow_{X_V}^ {\log(b_{\Sigma}(z)/\epsilon)})_*J_{V^{\le \epsilon}}$. The last property above also follows from the inspection of the aforementioned equations and flow invariance, as the flow of $X_W$ does not change the $y$.
\end{proof}

To complete the definition of $J_{W}$ within $(\epsilon-\delta, \infty)_{\sigma} \times Y \times \Sigma$, choose an extension so that
\be
\item $J_{W}R_{Y}$ satisfies the properties listed in Lemma \ref{Lemma:JRY},
\item $J_{W}|_{\xi_{Y}} = J_{\xi_Y}$, and
\item for each $\zeta \in T\Sigma$, $d\beta_{\Sigma}(\zeta, J\zeta) > 0$;
\item $J_W$ is tamed by $\rd \beta_W$.
\ee
This completes the definition on all of $J_{W}$. As a consequence, $J_W$ is an almost complex structure on $\completion{W}$, such that $J_W$ is tamed by $\rd \beta_W$ on $\completion{W}$ and $\alpha_\divSet$-tame on the positive cylindrical end, i.e. an almost complex structure for cobordisms used in \cite{BH-cylindrical,BH}.

\begin{proof}[Proof of \Cref{Lemma:FoliationProperties}]
    For $z\in \completion{\Sigma}$, and $v\in V^{\le \epsilon}$, we define 
    $$\leafInclusion_z(v) = \left(\Flow_{X_V}^ {\log(b_{\Sigma}(z)/\epsilon)}(v),z\right)\in \completion{V}\times \completion{\Sigma}$$
    where $b_{\Sigma}$ is extended to $\completion{\Sigma}$ by the linear function $\epsilon p$. It is clear from the definition of $J_W$, $\leafInclusion_z$ is holomorphic using $J_V$ on $V^{\le \epsilon}$. Now we extend the leaves by the image of 
    $$Y\times I \to \completion{W}, \quad (y,t)\mapsto (y,\gamma(t)),$$
    where $\gamma'(t)=-J_WR_Y$. Now since $J_WR_Y$ has a negative coefficient in $\partial_{\sigma}$, the path $\gamma$ always leaves the abstract extension region and outside this region, those leaves coincide with the foliation $\foliation_{\divSet}$ in \Cref{lemma:foliation_symp}. Since near the boundary of $\leafInclusion_z(V^{\le \epsilon})$, $-J_WR_Y$ is parallel to $\partial_{\sigma}$, two foliations patch together smoothly. 
\end{proof}

We use $\foliation_W$ to denote the foliation in \Cref{Lemma:FoliationProperties}, which on the positive cylindrical end coincides with $\foliation_{\divSet}$ in \Cref{lemma:foliation_symp}.

\section{Holomorphic curves lie in leaves}

In this section we provide a detailed analysis of contractible Reeb orbits in the contact hypersurface $\divSet \subset \completion{W}$. Then we show that holomorphic planes in $\completion{W}$ which are positively asymptotic to these orbits must be contained in leaves of the holomorphic foliation $\foliation_W$ described in the preceding section. Similarly we show that all $CH$-type holomorphic curves in $\completion{\divSet}$ which are asymptotic to only contractible orbits are contained in leaves of $\foliation_{\divSet}$.

\subsection{Detailed dynamics on $Y \times \Sigma^{\le \epsilon}$}\label{Sec:DetailedDynamics}

Here we study dynamics of the contractible closed orbits of the Reeb vector field $R_{\divSet}$ along the contact hypersurface $\divSet$, providing detailed models of neighborhoods of orbits and computations of Conley-Zehnder indices. According to Lemma \ref{Lemma:ContractibleOrbitLocation}, all such orbits are contained in the subset $\divSet_{\Sigma^{\le \epsilon}} \subset \divSet$ defined as the graph of a function $f_{\Sigma} \in \Cinfty(\Sigma)$ described in \S \ref{Sec:divSetSigmaDef}. We work in the domain of the map $\Phi^{\Sigma}$ described there along the subset $Y_{y} \times \Sigma^{\le \epsilon}_{z}$. 

\begin{notation}
Throughout this section we write $f=f_{\Sigma}$ with associated Hamiltonian vector field $\fX_{f} = \fX^{\Sigma}_{f_{\Sigma}}$ to simplify notation.
\end{notation}

The Reeb vector field $R_{\divSet}$ is described in Equation \eqref{Eq:ReebMSigma} and can alternatively be written
\begin{equation*}
R_{\divSet} = \left(f - \beta_{\Sigma}(\fX_{f})\right)^{-1}\left( R_{Y} - J_{\Sigma}\grad f\right).
\end{equation*}

Here we are using the almost complex structure $J_{\Sigma}$ on $\Sigma^{\le \epsilon}$ and $\grad$ is the gradient on $\Sigma$ with respect to the induces metric $d\beta_{\Sigma}(\ast, J_{\Sigma}\ast)$. We recall that our contact form is $\alpha_{\divSet}=f\alpha_{Y} + \beta_{\Sigma}$.

We assume a Morse-Smale setup so that $\grad f$ flow lines completely contained in $\Sigma^{\le \epsilon}$ have strictly positive Morse indices as $\rd f/\rd p<0$ along the boundary of $\Sigma^{\le \epsilon}$ in  \S \ref{Sec:divSetSigmaDef}. We also assume that near each $\zeta \in \Crit(f)$ we have a disk in $\Sigma$ with coordinate $z=(x, y)$ within which
\begin{equation}\label{Eq:MorseNearCritical}
\begin{gathered}
\beta_{\Sigma} = \half(-y\rd x+ x\rd y), \quad J_{\Sigma}\partial_{x} = \partial_{y}\quad f = f(\zeta) + Q_{\zeta}(x, y),\\
\implies \rd\beta_{\Sigma} = \rd x\wedge \rd y, \quad \xi_{M}|_{Y \times \{\zeta\}} = \xi_{Y} \oplus T_{\zeta}\Sigma
\end{gathered}
\end{equation}
for a quadratic form $Q_{\zeta}$ on the variables $x, y$. Because the function $\rd f$ is $\mathcal{C}^{1}$-small along $\Sigma^{\le \epsilon}$ and decreases as we approach the boundary, we can assume that
\be
\item there is a single $\zeta_{2}$ of $\ind_{\Morse}(\zeta_{2}) = 2$ and
\item there are $-1-\chi(\Sigma)$ critical points $\zeta_{1, k}$ of $\ind_{\Morse}(\zeta_{1, k})= 1$.
\ee
More specifically, we assume that
\be
\item near $\zeta_{2}$
\begin{equation*}
f=1 + 2\delta  - \half\epsilon(x^{2} + y^{2}), \quad \grad f = -\epsilon(x\partial_{x} + y\partial_{y}), \quad -\fX_{f}=\epsilon(-y\partial_{x} + x\partial_{y}),
\end{equation*}
\item near each $\zeta_{1, k}$
\begin{equation*}
f= 1 + \delta  - \half\epsilon(x^{2} - y^{2}), \quad \grad f = \epsilon(-x\partial_{x} + y\partial_{y}), \quad -\fX_{f}=\epsilon(y\partial_{x} + x\partial_{y}),
\end{equation*}
\ee
for $\epsilon, \delta > 0$ arbitrarily small. Then each closed orbit $\gamma_{Y}$ of $R_{Y}$ in $Y$ determines $-\chi(\Sigma)$ closed orbits of $R_{\divSet}$,
\begin{equation*}
\gamma_{2} = \gamma_{Y} \times \zeta_{2}, \quad \gamma_{1, k}=\gamma_{Y}\times \zeta_{1, k},
\end{equation*}
in $Y \times \Sigma^{\le \epsilon}$ whose actions are
\begin{equation*}
\int_{f\alpha + \lambda}\orbit_{2} = (1 + 2\delta)\int_{\alpha}\gamma_{Y}, \quad \int_{f\alpha + \lambda}\orbit_{1,k} = (1 + \delta)\int_{\alpha}\gamma_{Y}.
\end{equation*}
The key observations here are that when we project our Reeb vector field onto small $\disk \subset \Sigma$ neighborhoods of the critical points of $f$ we see an elliptic fixed point of $-\fX_{f}$ of slightly positive rotation about $\zeta_{2}$ and a hyperbolic fixed point about $\zeta_{1,k}$.

We compute the Conley-Zehnder indices and contact homology gradings of these orbits. In order to obtain $\Z$-valued indices and gradings, we must frame our orbits. Let $\framing_{Y}$ be a framing of $\xi_{Y}$ over $\gamma_{Y}$ from which we obtain a framing
\begin{equation*}
    \framing = (\framing_{Y}, \partial_{x})
\end{equation*}
of $\xi$ over each $\gamma_{2}$ and $\gamma_{1, k}$. By the additivity of the Conley-Zehnder index $\CZ$ with respect to direct sums together with standard computations of $\CZ$ for Hamiltonian fixed points on surfaces, we see that
\begin{equation*}
\CZ_{\framing}(\gamma_{1, k}) = \CZ_{\framing_{Y}}(\gamma_{Y}), \quad \CZ_{\framing}(\gamma_{2}) = \CZ_{\framing_{Y}}(\gamma_{Y}) + 1 + \left\lfloor \frac{\epsilon (1+2\delta)}{2\pi} \int_{\alpha}\gamma_{Y} \right\rfloor.
\end{equation*}
Consequently, contact homology indices $|\ast|_{\framing}$ computed with respect to our framings are
\begin{equation*}
|\gamma_{1, k}|_{\framing} = |\gamma_{Y}|_{\framing_{Y}} + 1, \quad |\gamma_{2}|_{\framing} = |\gamma_{Y}|_{\framing_{Y}} + 2
\end{equation*}
for a fixed $\gamma_{Y}$ whenever $\epsilon$ is sufficiently small.

\subsection{Change of coordinates and asymptotics of the $\leaf_{z}$}

We recall that in \S \ref{Sec:divSetSigmaDef}, we used a domain $[1, \infty) \times Y \times \Sigma^{\le \epsilon}$ with coordinates $(s, y, z)$ to define a map $\Phi^{\Sigma}$ into $[2\epsilon, \infty)_{\sigma} \times Y_{y} \times \completion{\Sigma} \subset \completion{W}$ via $\Phi^{\Sigma}(\sigma, y, z) = \Flow^{\log(s)}_{X_{W}}(f(z), y, z)$ implying that $(\Phi^{\Sigma})^{\ast}\beta_{W} = s(f\alpha_{Y} + \beta_{\Sigma})$ and that $X_{W} = s\partial_{s}$ on the coordinates of the domain.

We make a change of coordinates on the domain so that $J_{W}$ will have a standard form, replacing $s$ with $e^{s/C}$. Throughout the remainder of this section, we use instead a map
\begin{equation*}
\Psi^{\Sigma}: [0, \infty)_{s} \times Y_{y} \times \Sigma^{\le \epsilon} \rightarrow [2\epsilon, \infty)_{\sigma} \times Y_{y} \times \completion{\Sigma}, \quad \Psi^{\Sigma}(s,y,z) = \Flow_{X_{W}}^{s/C}(f(z), y, z).
\end{equation*}
This implies that on the domain of $\Psi^{\Sigma}$ our Liouville form and almost complex structure $J_{W}$ are given by
\begin{equation*}
\beta_{W} = e^{s/C}(f\alpha_{Y} + \beta_{\Sigma}), \quad J_{W}\partial_{s} = (R_{Y} - X_{f}), \quad J_{W}|_{T\Sigma} = J_{\Sigma}.
\end{equation*}
Using \eqref{Eq:JAlongMSigma}, $J_W$ on the submanifold $Y\times \Sigma^{\le \epsilon}\subset W^{\Box}$ with contact form $f_{\Sigma}\alpha_Y+\beta_\Sigma$, $J_W$ on the graph of $f_{\Sigma}$ identifies with $J_{\Sigma}$ on $T_{\Sigma}$ using the $f_{\Sigma}$. Therefore $J_W$ under the pullback of $\Psi^{\Sigma}$ is given by the above formula. 

Using this coordinate system, the leaves $\leaf_{z}$ of the holomorphic foliation $\foliation_\divSet,\foliation_W$ of \Cref{lemma:foliation_symp,Lemma:FoliationProperties} have tangent spaces
\begin{equation*}
T\leaf_{z} = \xi_{Y} \oplus \langle R_{Y}, J_{W}R_{Y}\rangle  = \xi_{Y} \oplus\langle \partial_{s} + \grad f, R_{Y}\rangle .
\end{equation*}
This formula together with already-established observations, we restate the following property of the asymptotic of $\leaf$ in \Cref{lemma:foliation_symp}.

\begin{lemma}\label{Lemma:LeafAsymptotics}
In the $\Psi$ coordinate system, the intersection of each $\leaf_z$ with $\im \Psi$ can be parameterized by a map solving a single variable ordinary differential equation
\begin{equation*}
\Upsilon_{z}: [0, \infty) \times Y \rightarrow [0, \infty) \times Y \times \Sigma^{\le \epsilon}, \quad \Upsilon_{z}(s, y) = (s, y, \upsilon_{z}(s)), \quad \frac{\partial \upsilon_{z}}{\partial s} = \grad f(s)
\end{equation*}
with initial condition $\upsilon_{z}(0)$ uniquely determined by $z$. So in this coordinate system the projection $\pi_{\Sigma}\leaf_{z} \subset \Sigma^{\le \epsilon}$ is a $\grad f$ flow line. As $\grad f$ points into $\Sigma^{\le \epsilon}$ along its boundary, for each $z \in \Sigma$ the positive asymptotic of the leaf is given by
\begin{equation*}
\lim_{s \rightarrow \infty}\pi_{Y \times \Sigma^{\le\epsilon}}\Upsilon_{z}(\{s\} \times Y) = Y \times \{ \zeta\}, \quad \zeta \in \Crit(f).
\end{equation*}
\end{lemma}

\subsection{Holomorphic asymptotics near critical point orbits}\label{Sec:Asymptotics}

We consider holomorphic maps $u$ into either $\completion{W}$ or $\completion{\divSet}$ which positively asymptotic to our contractible $R_{\divSet}$ orbits $\gamma_{\zeta} := \gamma_{Y} \times \zeta$ for $\zeta \in \Crit(f)$. In this subsection we describe asymptotic Fourier expansions about ends of such $u$ in the ``$L$-simple'' style of \cite{BH, BH-cylindrical}.

Let $N_Y(\gamma_Y) \subset Y$ be a $\Circle \times \disk^{n-1}$ ($\disk$ is a disk in $\C$) neighborhood of an embedded $\gamma_{Y}$ in $Y$. A $N_{Y}(\gamma_Y) \times \disk \subset Y \times \Sigma^{\le \epsilon}$ neighborhood centered about each $\gamma_{\zeta} = \gamma_{Y} \times \zeta$ in $Y \times \disk$ is exactly in the form local described in \cite{Avdek:Hypersurface,BH}. For the purpose of studying asymptotics of holomorphic curves to multiply-covered $\gamma_{\zeta}$ which are $k$-fold coverings of an embedded $\gamma'_{\zeta}$, we can take the $N(\gamma_{\zeta})$ to be a $k$-fold covering of $N(\gamma'_{\zeta})$ and consider lifts $u$ -- restricted to neighborhoods of punctures in the domain -- to this $N(\gamma_{\zeta})$.

Since $u$ is positively asymptotic to $\gamma_{\zeta}$, we can assume that for $s_{0} > 0$ sufficiently large there is cylindrical end $[s_{0}, \infty)_{p} \times \Circle_{q}$ of the domain of $u$ send via $u$ to $[s_{0}, \infty) \times N_{Y}(\gamma_{Y}) \times \disk \subset [s_{0}, \infty)\times Y \times \Sigma^{\le \epsilon}$. Hence we can write the restriction of $u$ to this positive cylindrical end as
\begin{equation*}
u: [s_{0}, \infty)_{p} \times \Circle_{q} \rightarrow [s_{0}, \infty)_s \times Y \times \Sigma^{\le \epsilon}, \quad u(p,q) = (s(p, q), u_{Y}, u_{\disk})
\end{equation*}
Possibly after increasing $s_{0}$ we can further assume that $s(p, q) = p$ using the standard complex structure $\domainJ \partial_{p} = \partial_{q}$ on the domain. This is a consequence of the Riemann mapping theorem, cf. \cite[\S 2]{Avdek:Hypersurface}. Then the fact that $u$ is holomorphic implies that
\begin{equation*}
\delbar_{J_{Y}}(s, u_{Y}) = 0, \quad \delbar u_{\disk} - \left( X_{f} \otimes ds \right)^{0, 1} = 0.
\end{equation*}
It follows that $u_{\disk}$ admits a Fourier expansion of the form
\begin{equation}\label{Eq:LocalFourier}
u_{\disk} = \sum_{\eigenvalue_{k} < 0}\sum_{k,l}c_{k,l}e^{\eigenvalue_{k}p}\eigenfunction_{k,l}(q), \quad c_{k,l} \in \R.
\end{equation}
Here $(\eigenvalue_{k},\eigenspace_{k})$ is an eigenvalue-eigenspace decomposition of the asymptotic operator
\begin{equation}\label{Eq:NormalAsymptoticOp}
\AsymptoticOp_{\zeta}: \Sobalev^{1, 2}(\Circle, \C) \rightarrow \Ltwo(\Circle, \C), \quad \AsymptoticOp_{\zeta} = -J_{\Sigma}\left(\frac{\partial}{\partial_{q}} + (-1)^{\ind_{\Morse}(f, \zeta)}\epsilon y\partial_{y} - \epsilon x\partial_{x}\right)
\end{equation}
and the $\eigenfunction_{k, l}$ are a basis of eigenfunctions spanning each $\eigenspace_{k}$. Note that the condition $\eigenvalue_{k} < 0$ in Equation \eqref{Eq:LocalFourier} is required in order that we get asymptotic convergence to the orbit as $p \rightarrow \infty$.

We index the eigenvalues of $\AsymptoticOp_{\zeta}$ so that $\eigenvalue_{k + 1} > \eigenvalue_{k}$ with $\sgn(\eigenvalue_{k})=\sgn(k)$. Since $\epsilon \neq 0$, the kernel of $\AsymptoticOp_{\zeta}$ is zero, so the indexing set for the $k$ is $\Z_{\neq 0}$. The operators $\AsymptoticOp_{\zeta}$ are identical to those associated to elliptic and positive hyperbolic orbits in $L$-simple contact $3$-manifolds which are explicitly worked out in \cite[\S 4]{BH-cylindrical}. We summarize the key properties of the $\eigenvalue_{k}, \eigenspace_{k}, \eigenfunction_{k,l}$ for $k < 0$, which we need for the analysis of positive asymptotics of holomorphic curves, as follows:
\be
\item Each $\eigenfunction_{k,l}$ sends $\Circle$ to $\disk \setminus \{ 0 \}$ and so has a winding number $\wind_{k}\in \Z$ which depends only on $k$.
\item For both $\ind_{\Morse}(f, \zeta)=1, 2$, the condition $k < -1$ implies $\wind_{k} < 0$.
\item For both $\ind_{\Morse}(f, \zeta)=1, 2$, the largest negative eigenvalue is $\eigenvalue_{-1}=-\epsilon$.
\item When $\ind_{\Morse}(f, \zeta)=2$, the elliptic case, the $\eigenvalue_{-1}$ eigenspace has $\dim=2$ and is spanned by constant eigenfunctions $\eigenfunction_{-1, x} = (1, 0)$ and $\eigenfunction_{-1, y}(0, 1)$, both having $\wind_{-1}=0$.
\item For $\ind_{\Morse}(f, \zeta)=1$, the hyperbolic case, the $\eigenfunction_{-1}$ eigenspace has $\dim=1$ spanned by $\eigenfunction_{-1} = (1, 0)$, also having $\wind_{-1} = 0$.
\ee

Observe that if $c_{k,l} = 0$ for all $k < -1$ in Equation \eqref{Eq:LocalFourier}, then $u_{\disk}$ is exactly a gradient flow line of $f$. Indeed, we will then have $x_{0}, y_{0} \in \R$ for which
\begin{equation*}
u_{\disk}(p, q) = \begin{cases}
    e^{-\epsilon p}(x_{0}, 0) & \ind_{\Morse}(f,\zeta)=1,\\
    e^{-\epsilon p}(x_{0}, y_{0}) & \ind_{\Morse}(f,\zeta)=2.
\end{cases}
\end{equation*}
If $u_{\disk}$ has this form, then the half cylindrical end of $u$ is contained in some leaf $\leaf_{z}$ of $\foliation$ by Lemma \ref{Lemma:LeafAsymptotics}. By the analytic continuation property for holomorphic curves \cite[\S 2.2]{MS:Curves}, it would then follow that all of $\im u$ is contained in the same $\leaf_{z}$. In summary we have proved the first statement of the following lemma. The proof of the second statement is identical.

\begin{lemma}\label{Lemma:LeafFourierCriterion}
A holomorphic map $u$ from a punctured Riemann surface into $\completion{W}$ (or $\completion{\divSet}$) having a puncture positively asymptotic to an orbit $\orbit_\zeta, \zeta\in \Crit(f)$ is contained in a leaf of $\foliation_W$ (respectively, $\foliation_{\divSet}$) iff at that positive end the Fourier expansion of Equation \eqref{Eq:LocalFourier} has $c_{k,l}=0$ for all $k < -1$.
\end{lemma}

\subsection{Holomorphic curves lie in leaves}
We seek to show that every holomorphic curve of contact homology type in $\completion{W}$ (or $\completion{\divSet}$) which is asymptotic to contractible orbits of the form $\gamma_{Y} \times \zeta, \zeta \in \Crit(f)$ is contained in a leaf of $\foliation_W$ (respectively, $\foliation_{\divSet}$). This will be a consequence of Lemma \ref{Lemma:LeafFourierCriterion}.

We first consider $u: \C \rightarrow \completion{W}$ and look at the positive end, applying the notation of the preceding subsection. Suppose that $u$ is not contained in a leaf of $\foliation_W$, meaning that the Fourier decomposition of $u_{\disk}$ is such that not all $c_{k,l}$ are zero for $k<-1$. Let $k_{0}<-1$ be the largest index for which there is a $c_{k_{0}, l} \neq 0$. We write
\begin{equation}\label{Eq:UdiskAsymptotics}
u_{\disk} = e^{-\epsilon p}(x_{0}, y_{0}) + e^{\eigenvalue_{k_{0}}p}\sum_{l}c_{k_{0},l}\eigenfunction_{k_{0}, l}(q) + u_{\disk, \hot}, \quad u_{\disk, \hot} = \sum_{k < k_{0},l}c_{k,l}e^{\eigenvalue_{k}p}\eigenfunction_{k}(q).
\end{equation}
where the domain of $u$ is $[s_{+}, \infty) \times \Circle_{q}$. Let $\leaf$ be the leaf of $\foliation$ passing through the point $\{ s_{+} \} \times Y \times \{ e^{-\epsilon s_+}(x_{0}, y_{0})\}$, so that where $\disk \subset \Sigma^{\le \epsilon}$ is centered about $\zeta \in \Crit(f)$ we have
\begin{equation*}
\leaf \cap ([s_{+}, \infty)_{s} \times Y \times \disk) = \left\{ (s, Y, e^{-\epsilon s}(x_{0}, y_{0}))\ :\ s \geq s_{+}\right\}.
\end{equation*}

Let $\widetilde{\Sigma} \simeq \R^{2}$ be the universal cover of $\completion{\Sigma}$ and let $\widetilde{\zeta}$ be a point in the cover over $\zeta \in \Sigma$. We can choose a lift of $\widetilde{\leaf}$ of $\leaf$ to $\completion{V} \times \widetilde{\Sigma}$ whose positive end converges to $\{ \infty\} \times Y \times \widetilde{\zeta}$ along its ideal boundary and a lift $\widetilde{u}$ of $u$ converging to the orbit $\widetilde{\gamma}_{\widetilde{\zeta}} = \gamma \times \widetilde{\zeta}$. Such lifts exist by the fact that the domain $\C$ of $u$ is simply connected and consideration of the inclusion $\pi_{1}(\leaf) \rightarrow \pi_{1}(\completion{W})$ being injective. The almost complex structure $J_{W}$ lifts to the cover so that $\widetilde{\leaf}$ is a holomorphic submanifold and $\widetilde{u}$ is a holomorphic map.

We can express the cover as $\leaf \times \R^{2}$ so that \be 
\item the $\R^{2}$ fiber agrees with the $\disk$ factor of $[s_{+}, \infty) \times Y \times \disk$ over the half-cylindrical end $[s_{+}, \infty) \times Y$ of $\completion{V}$ near $\leaf$ and
\item the map $\disk \rightarrow \R^{2}$ at $\{ s_{+} \} \times Y$ is affine with $(x_{0}, y_{0}) \in \disk$ is sent to $0 \in \R^{2}$.
\ee
Let $V_{\leq s_{+}}$ be the complement of this half-cylindrical end in $\leaf$. For $s_{+}$ large, the $c_{k_{0}}$ term in Equation \eqref{Eq:UdiskAsymptotics} dominates the $u_{\disk, \hot}$ term, due to the comparative exponential decay factors $e^{\eigenvalue_{k_0}p}$. So for $s_{+}$ large, we get a loop
\begin{equation*}
\begin{gathered}
L:\Circle_{q} \rightarrow \R^{2}\setminus \{0\}, \quad L(q) = u(s_{+},q)= e^{\eigenvalue_{k_{0}}s_{+}}\sum_{l}c_{k_{0},l}\eigenfunction_{k_{0},l}(q) + u_{\disk, \hot}\\
\wind(L) = \wind_{k_{0}} < 0.
\end{gathered}
\end{equation*}
This winding number computes exactly the intersection number of $\widetilde{u}$ restricted to $\C \setminus \{ p > s_{+}\}$ with $V_{\leq s_{+}} = V_{\leq s_{+}} \times \{0\} \subset \completion{V} \times \R^{2}$. This intersection number is well-defined because $\widetilde{u}$ and $V_{\leq s_{+}}$ are disjoint along their boundaries. Since $V_{\leq s_{+}}$ is a holomorphic submanifold and $u$ is holomorphic, this intersection number must be non-negative. This contradicts the fact that the $c_{k_{0},l}$ are not all zero. Therefore we have proved the lemma.

\begin{lemma}\label{Lemma:PlaneInLeaf}
For each holomorphic plane $u: \C \rightarrow \completion{W}$ positively asymptotic to an $R_{\divSet}$ orbit $\gamma_{\zeta}$ for $\zeta \in \Crit(f)$, the image of $u$ is contained in a leaf $\leaf$ of the foliation $\foliation_{W}$.
\end{lemma}

A similar analysis applies to holomorphic curves in the symplectization $\completion{\divSet}$.

\begin{lemma}\label{Lemma:CurveInLeaf}
Let $u$ be a holomorphic curve in $\completion{\divSet}$ which is positively asymptotic to a $\gamma_{\zeta}$ for some $\zeta \in \Crit(f)$ and negatively asymptotic to some $\gamma_{i,\zeta_i} = \gamma_{i} \times \zeta_{i}$ for $\zeta_{i} \in \Crit(f)$ and Reeb orbits $\gamma_{i}$ of $Y$. If all the $\gamma_{i}$ are contractible in $\divSet$ then the image of $u$ is contained in a leaf of $\foliation_{\divSet}$.
\end{lemma}

\begin{proof}
Identify the cylindrical end of $\completion{W}$ with $[0,\infty)_{s}\times \divSet$. For each $s_{+} > 0$ the complement of $[s_{+}, \infty) \times \divSet$ in $\completion{W}$ is a compact symplectic manifold which we will call $W_{\leq s_{+}}$.

For constants $s_{\pm}$ of the form $s_{-} \ll 0 \ll s_{+}$ let $u_{s_{\pm}} = u \cap \left([s_{-}, s_{+}] \times \divSet\right)$. The $s_{-}$ boundary of $u_{s_{\pm}}$ is a collection of circles $\gamma_{i, -}$ in $\divSet$ which are as close as we like to the $\{-\infty\} \times \gamma_{i}$ by making $|s_{-}|$ large. Since the $\gamma_{i}$ are contractible, we can find a collection of disks $D_{i} \subset \{s_{-}\}\times \divSet$ which bound the $\gamma_{i, -}$. Let $u_{s_{\pm},D}$ be the union of $u_{s_{\pm}}$ with the $D_{i}$, which is topologically a disk.

As in the study of holomorphic planes in $\completion{W}$, consider the lift $\widetilde{u}_{s_{\pm}, D}$ of $u_{s_{\pm}, D}$ to $\completion{V} \times \widetilde{\Sigma}$, by viewing $[s_{-},s_{+}] \times \divSet$ as being contained in $\completion{W}$. Let $\leaf$ be a $\completion{V}$ leaf of $\foliation_{\divSet}$ positively asymptotic to $\{\infty\} \times Y \times \{ \zeta\}$, isomorphic to $\completion{V}$, and being completely contained in $(s_{-}, \infty) \times \divSet$. Write $\leaf_{\leq s_{+}} = \leaf \cap [s_{-}, s_{+}]\times \divSet$. Then we can similarly lift $\leaf_{\leq s_{+}}$ to a $\widetilde{\leaf}_{\leq s_{+}} \subset \completion{V}\times \widetilde{\Sigma}_{\leq s_{0}}$. As in the case of holomorphic planes, we require the lifts $\widetilde{u}_{s_{-},s_{+}, D}$ and $\widetilde{\leaf}_{\leq s_{+}}$ to have positive asymptotics tending to the same $\{ \infty \} \times Y \times \{ \widetilde{\zeta}\}$ as $s_{+} \rightarrow \infty$.

Now we study intersections of $\widetilde{u}_{s_{\pm}, D}$ and $\widetilde{\leaf}_{\leq s_{+}}$ for $s_{+} \gg 0$ and with $\leaf$ varying, seeking to show that all of the $c_{k,l}=0$ whenever $k < -1$ in the Fourier expansion of the positive end of $u$. As in our previous analysis, suppose that $u$ is not contained in a leaf of $\foliation_{\divSet}$ and take the maximal $k_{0} < -1$ for which $\sum_{l} c_{k_{0}}\eigenfunction_{k_{0}, l} \neq 0$. We are then again in the situation of Equation \eqref{Eq:UdiskAsymptotics}. As $\sum_{l} c_{k_{0},l}\zeta_{k_{0}, l}(q)$ is never zero and $u_{\disk, \hot}$ is negligible, we can find a $(x_{1}, y_{1}) \in \disk$ as close as we want to $e^{-\epsilon s_{0}}(x_{0}, y_{0})$, for which
\be
\item $(x_{1}, y_{1})$ and $e^{-\epsilon s_{+}}(x_{0}, y_{0})$ are in the same connected component of $\disk \setminus \pi_{\disk}(\text{$s_1$-boundary of $u_{s_{\pm}}$})$, and
\item $\leaf$ contains $\{s_{+}\} \times Y \times \{(x_{1},y_{1})\}$.  Since $\completion{V}$-leaves are dense, we can pick a generic $(x_1,y_1)$, such that $\leaf$ is a $\completion{V}$-leaf.
\ee
In other words, the loop $L(q)$ winds around the point $(x_{1}, y_{1}) \in \disk$ a total of $\wind_{k_{0}}<0$ times which computes the intersection number of $\widetilde{u}_{s_{\pm}, D}$ and $\widetilde{\leaf}_{s_{+}}$. As before this intersection number is $\wind_{k_{0}}$. 
So the intersection number is $\wind_{k_{0}}<0$ and $\widetilde{\leaf}_{s_{1}}$ only intersects $\widetilde{u}_{s_{\pm}, D}$ where $\widetilde{u}_{s_{\pm}, D}$ is holomorphic, since $\leaf_{\leq s_{+}}$ lives in the collar $[s_{-}, s_{+}] \times \divSet$, while the $D_{i}$ are contained in $W_{s_{+}}$.

This intersection number cannot be negative, yielding a contradiction. We conclude that only the $k=-1$ Fourier coefficient of $u$ can be non-zero. Hence $u$ must be tangent to a leaf of $\foliation_{\divSet}$ by Lemma \ref{Lemma:LeafFourierCriterion}.
\end{proof}

\begin{rmk}
    Alternatively, \Cref{Lemma:PlaneInLeaf,Lemma:CurveInLeaf} can be proved using the Siefring intersection theory considered by Moreno and Siefring in the presence of a holomorphic foliation \cite{moreno2019holomorphic}. Moreover, precisely using the natural framing from $\Sigma$ to push off period orbits $\gamma_{\zeta}$, the linking number of the push-off of $\gamma_{\zeta}$ with any leaf is zero, which is well-defined for contractible orbits in the case we consider. Now since the normal Conley-Zehnder indices, as explained in \S \ref{Sec:Asymptotics} is $\ind_{\Morse}(f, \zeta)-1=0,1$. The intersection number \cite[\S 2]{moreno2019holomorphic} (or see \cite[Definition 5.1]{zbMATH07821993}) $u\ast \leaf$ with any leaf is $0$, then by \cite[ Corollary 2.3 and Theorem 2.5]{moreno2019holomorphic}, $u$ must be contained in a leaf.
\end{rmk}

\section{Fredholm theorem for curves in leaves}\label{Sec:Fredholm}

We quickly work out the basic Fredholm theory for holomorphic curves in $\completion{W}$ and $\completion{\divSet}$. The moduli spaces containing these curves will be more fully described in \S \ref{Sec:ModSpaceDefs}.

Let $u: \C \rightarrow \completion{W}$ be a holomorphic plane asymptotic to a $\orbit_{\zeta} = \orbit \times \{\zeta\} \subset Y \times \Sigma$ for a $\zeta \in \Crit(f)$. By Lemma \ref{Lemma:PlaneInLeaf} we can express $u$ as
\begin{equation*}
    u = \leafInclusion_{z} \circ u_{\completion{V}}, \quad z\in \completion{\Sigma}, \quad u_{\completion{V}}: \C \rightarrow \completion{V}
\end{equation*}
where $\leafInclusion_{z}: \completion{V} \rightarrow \completion{W}$ is the leaf inclusion map of Lemma \ref{Lemma:FoliationProperties} with $u_{\completion{V}}$ holomorphic and positively asymptotic to $\gamma$. Therefore $\im u \subset \leaf_{z} = \im \leafInclusion_{z}$, a leaf which is positively asymptotic to $\{ \infty\} \times Y \times \{\zeta\}$ in the sense of Lemma \ref{Lemma:LeafAsymptotics}.

Write $T\foliation_{W} \subset T\completion{W}$ for the union of the tangent spaces of the leaves of $\foliation_{W}$. If $u_{\completion{V}}$ is any -- not-necessarily holomorphic -- map of a plane into $\completion{V}$ then we have a commutative diagram for which the vertical maps are isomorphisms defining an operator $\Dlinearized_{\foliation}$,
\begin{equation*}
\begin{tikzcd}
\Omega^{0}\left(u^*_{\completion{V}}T\completion{V}\right) \arrow[d]\arrow[r, "\Dlinearized_{u_{\completion{V}}}"] & \Omega^{0, 1}\left(u^*_{\completion{V}}T\completion{V}\right) \arrow[d]\\
\Omega^{0}\left((\leafInclusion_{z}\circ u_{\completion{V}})^{\ast}T\foliation_{W}\right) \arrow[r, dashed, "\Dlinearized_{\foliation}"] & \Omega^{0, 1}\left((\leafInclusion_{z}\circ u_{\completion{V}})^{\ast}T\foliation_{W}\right).
\end{tikzcd}
\end{equation*}
The integrability of $T\foliation_{W}$ implies that in fact
\begin{equation*}
\Dlinearized_{\foliation}: \Omega^{0}_{\foliation} \rightarrow \Omega^{0, 1}_{\foliation}, \quad \Omega^{\ast}_{\foliation} = \Omega^{\ast}(u^{\ast}T\foliation_{W})
\end{equation*}
is well-defined for any $u: \C \rightarrow \completion{W}$, even when $\im u$ is not contained in a leaf of $\foliation_{W}$.

Let $T\foliation^{\perp}_{W} \subset T\completion{W}|_{\leaf_{z}}$ be a complement to $T\foliation_W$. This is a trivial $\R^{2}$ bundle. Indeed, since $\completion{\Sigma}$ is an open Riemann surface, $T\completion{\Sigma}$ admits a global trivialization. Given $\partial_{x}, \partial_{y}$ spanning $T_{z}\completion{\Sigma}$ we could choose $T\foliation^{\perp}_{W}|_{\leaf_{z}}$ to be spanned by nowhere vanishing sections $\partial_{x}\leafInclusion(z,v), \partial_{x}\leafInclusion(z,v)$, so a choice of trivialization of $T\completion{\Sigma}$ induces a global trivialization of $T\foliation_{W}^{\perp}$. We will assume that a trivialization is chosen so we are identifying $\R^{2}$ with $\Span(\partial_{x}, \partial_{y})$ along $[s_{0}, \infty)\times Y \times \Sigma^{\le \epsilon}$ for $s_{0} \gg 0$ and where the $xy$-coordinates are on the $\disk \subset \Sigma$ centered about $\zeta$ described in \S \ref{Sec:DetailedDynamics}. We also assume that $T\foliation^{\perp}_{W}$ is preserved by our almost complex structure $J_{W}$.

At a given map $u: \C \rightarrow \completion{W}$, we can split sections of $u^{\ast}T\completion{W}$ into $u^{\ast}T\foliation_{W}$ and $u^{\ast}T\foliation_{W}^{\perp}$ summands. Therefore the linearization $\Dlinearized_{u}$ of $u$ can be written as a triangular block matrix
\begin{equation}\label{Eq:DBlockMatrix}
\Dlinearized_{u} = \begin{pmatrix}
    \Dlinearized_{\foliation} & \Dlinearized_{ur}\\
    0 & \Dlinearized_{\perp}
\end{pmatrix} \implies \ind \Dlinearized_{u} = \ind \Dlinearized_{\foliation} + \ind \Dlinearized_{\perp}
\end{equation}
with $\Dlinearized_{ur}=0$ if the image of $u$ is contained in a leaf. Here $\Dlinearized_{\foliation}$ is as above and $\Dlinearized_{\perp}$ is defined
\begin{equation*}
\Dlinearized_{\perp}: \Omega^{0}_{\perp} \rightarrow \Omega^{0, 1}_{\perp}, \quad \Omega^{\ast}_{\perp} = \Omega^{\ast}(u^{\ast}T\foliation_{W}^{\perp})
\end{equation*}
That $\Dlinearized_{u}$ is upper-triangular follows from the parametrization of the foliation in \Cref{Lemma:FoliationProperties}. The upper-right term $\Dlinearized_{ur}$ captures the variation of $J_{W}$ on the normal directions, which preserves the leaf tangent space.

The linearized operator $\Dlinearized_{\perp}$ is a real linear Cauchy-Riemann operator on a trivial complex line bundle over the domain $\C$ of $u$ and is exactly as studied in \cite[Proposition 2.2]{AutomaticTransversality}. Since $\ind_{\Morse}(f, \zeta)$ computes the Conley-Zhender index at the positive end of $\Dlinearized_{\perp}$ which is described by Equation \eqref{Eq:NormalAsymptoticOp}, and this Morse index is $\geq 1$ by the assumptions of \S \ref{Sec:DetailedDynamics}, \cite{AutomaticTransversality} tells us that $\Dlinearized_{\perp}$ is always surjective, having
\begin{equation}\label{Eq:NormalPlaneAutomaticTransversality}
\dim \ker \Dlinearized_{\perp} = \ind \Dlinearized_{\perp} = \ind_{\Morse}(f, \zeta) \geq 1.
\end{equation}

By the triangular form of $\Dlinearized_{u}$ and the surjectivity of $\Dlinearized_{\perp}$ it follows that $\Dlinearized_{u}$ is surjective iff $\Dlinearized_{\foliation}$ is surjective. So assuming transversality for the $\Dlinearized_{\foliation}$, there can be no rigid ($\ind=0$) holomorphic planes positively asymptotic to $\orbit_{\zeta}$ and we would have $\aug_W(\orbit_{\zeta}) = 0$. However the surjectivity for $\Dlinearized_{\foliation}$ cannot be assumed even for generic $J_{V}$ on $\completion{V}$ due to the usual transversality problems of SFT such as multiple covers.

We will therefore have to consider perturbed holomorphic curves consisting of gluings of multi-level buildings as in \cite{BH}. This necessitates consideration of holomorphic curves $u$ in the symplectization $\completion{\divSet}$ of the ideal boundary $\divSet$ of $\completion{W} = \completion{V} \times \completion{\Sigma}$.

In order that such $u$ in $\completion{\divSet}$ fits together into a building whose topological type is a plane, all of the ends of $u$ are asymptotic to contractible orbits, satisfying the hypotheses of Lemma \ref{Lemma:CurveInLeaf}. Then $u$ is contained in a leaf of the holomorphic foliation $\foliation_{\divSet}$ and we can again write
\begin{equation*}
u = \leafInclusion_{\leaf} \circ u_{\leaf}
\end{equation*}
for a leaf inclusion map $\leafInclusion_{\leaf}: \leaf \rightarrow  \completion{\divSet}$ of some leaf $\leaf$ of $\foliation_{\divSet}$ and a holomorphic map $u_{\leaf}$ of the domain of $u$ into $\leaf$. The linearized operator for $u$ again has a triangular form as in Equation \eqref{Eq:DBlockMatrix} with $\Dlinearized_{\foliation}$ the linearization of $u_{\leaf}$ and $\Dlinearized_{\perp}$ the linearization of the normal bundle. By the characterization of leaves in \Cref{lemma:foliation_symp}, such a $u$ is positively asymptotic to some $\orbit_{\zeta} = \gamma \times \{\zeta\}$ and negatively asymptotic to some $\orbit_{\zeta, i} = \orbit_{i} \times \{ \zeta_{-}\}, i=1,\dots,m_{-}$ with $\zeta, \zeta_{-}\in \Crit(f) \subset \Sigma^{\le \epsilon}$. Therefore
\begin{equation}\label{Eq:NormalNotPlaneIndex}
\ind \Dlinearized_{\perp} = \ind_{\Morse}(f, \zeta) - m_{-}\ind_{\Morse}(f, \zeta_{-}).
\end{equation}
From \Cref{lemma:foliation_symp}, the specific cases are as follows:
\be
\item If $\leaf$ is a $\completion{V}$ leaf then there are no negative asymptotics, $u$ is a plane and we are in the situation of Equation \eqref{Eq:NormalPlaneAutomaticTransversality}, so that automatic transversality applies to $\Dlinearized_{\perp}$.
\item If $\leaf$ is a flow-line leaf of the form $\R_{s} \times Y \times \{\zeta\}$, meaning $\zeta=\zeta_{-}$, then $\ind \Dlinearized_{\perp} = (1-m_{-})\ind_{\Morse}(\zeta)$.
\item If $\leaf$ is a flow-line leaf positively asymptotic to $Y\times \{\zeta_{2}\}$ where $\zeta_{2} \in \Sigma^{\le \epsilon}$ is the unique critical point of $\ind_{\Morse}=2$ and $\zeta_{-}$ is a one of the $\zeta_{1, k}$ of $\ind_{\Morse}=1$, then $\ind \Dlinearized_{\perp} = 2-m_{-}$
\ee
So it is clear that even when $\Dlinearized_{\foliation}$ is surjective, $\ind \Dlinearized_{\perp}$ can be negative, implying that $\Dlinearized_{u}$ will often not be surjective. This entails that the $\ind=0$ holomorphic buildings we seek to count will in general be composed of levels having indices of arbitrary $\ind$.

\section{Vanishing of $\aug_{W}$}\label{Sec:VanishingAug}

\begin{theorem}\label{thm:no_augmentation}
Let $\Sigma$ be a surface with boundary that is not disk and $V$ a Liouville domain. Then there exists a sequence of contact forms $\alpha_{\divSet, k}$ on $\divSet:=\partial(\Sigma\times V)$ and $\alpha_{\divSet,k}$-tame almost complex structures $J_{W, k}$ on $\completion{W}=\completion{V} \times \completion{\Sigma}$, and choices of perturbations such that the following conditions hold:
\be
\item The $\alpha_{\divSet, k}$ are obtained by $\beta_{W}$ restricted to contact hypersurfaces $\divSet$ in $W$ as in \S \ref{Sec:Foliation} which converge in $\mathcal{C}^{1}$\footnote{It also suffices to work with a decreasing sequence of contact forms $\alpha_{\divSet,k}$.}. The $J_{W, k}$ are also as described in \S \ref{Sec:Foliation}.
\item There are period thresholds $L_{k} \rightarrow \infty$ such that the augmentation $\aug_{W}$ sending closed Reeb orbits of period up to $L_k$ in $(\divSet,\alpha_{\divSet,k})$ to $\Q$ by counting perturbed holomorphic planes in $\completion{W}$ using the $J_{W, k}$ is zero.
\ee
\end{theorem}

As mentioned above, the theorem applies even when moduli spaces of holomorphic planes $u_{\leaf_{W}}:\C \rightarrow \leaf$ whose targets are the leaves $\leaf$ of the foliation $\foliation_{W}$ are not transversely cut out.

We will use the Kuranishi framework of \cite{BH}, summarized in \cite[\S 10]{Avdek:Hypersurface}. The problem of counting holomorphic planes asymptotic to $\orbit$ in the present context is technically similar to the counting of holomorphic planes contributing to the contact homology differential for neighborhoods of convex hypersurfaces in \cite{Avdek:Hypersurface}. As in \cite{Avdek:Hypersurface}, counting of holomorphic curves in the symplectization $\completion{\divSet}$ contributing to the contact homology differential for $M$ should in general require obstruction bundle gluing calculations, but that is beyond the scope of the present text.

To compute $\aug_{W}$, we need to count perturbed SFT buildings of index $0$, whose components are given by maps into the leaves of $\foliation_{\divSet}$ and $\foliation_{W}$, and which become a plane when glued. Heuristically, the buildings can be thought of as buildings in $\completion{V}$ mapped into $\completion{W}$ by the leaves. Note that the normal linearized operators $\Dlinearized_{\perp}$ will have a total index of $1$ or $2$ depending on the critical point $\zeta \in \Sigma$, when the positive orbit is a $\orbit_{\zeta}$. The leaf linearized operators $\Dlinearized_{\foliation}$ must then have a total index of $-1$ or $-2$. We will show that there is a Kuranishi perturbation scheme for which $\aug_{W}\orbit_{\zeta}=0$ for all such $\orbit_{\zeta}$ as a consequence of $\ind \Dlinearized_{\foliation}=-1,-2$. Note that we are in a rather clean situation, as the foliation is parameterized in a transparent way. The following subsections formalize this argument. 

\subsection{Action bounds and contractible orbits}
Computing $CH$ using the framework of \cite{BH} requires that we consider sequences of action bounds $L_{k} \rightarrow \infty$ and closed Reeb orbits of action $\leq L_{k}$ associated to contact forms $\alpha_{\divSet, k}$ which $\mathcal{C}^{1}$-converge as $m\rightarrow \infty$. Using this (as well as additional geometric data such as almost complex structures and perturbation data), we define and compute $CH^{\leq L_{k}}$. Finally, we take a direct limit to define $CH$. Throughout the remainder of the section, the $k$ subscript will be ignored.

For the purpose of establishing that Bourgeois contact manifolds are tight, we consider the contact homology of the neighborhood $\R \times \Circle \times M$ of the Bourgeois convex hypersurface $\hypersurface = \Circle \times M$. According to the construction of contact forms in \cite{Avdek:Hypersurface} used to compute the contact homology of a neighborhood of a convex hypersurface, all orbits we must consider are contained in the diving set, which in the present context will be $\divSet = \partial W =\partial (V\times \Sigma)$ with $\Sigma$ is an annulus. Using the Reeb vector field $R_{\divSet}$ on $\divSet$ -- which is determined by the functions $f_{V}$ and $f=f_{\Sigma}$ -- all contractible Reeb orbits (which could contribute to augmentations for the fillings of $\divSet$ associated to the positive and negative regions of the convex hypersurface) are contained in the region $Y \times \Sigma^{\le \epsilon}$: Outside this region the $\partial_{q}$ summand of $R_{\divSet}$ is non-zero, so all Reeb orbits are non-contractible.

By considering sequences of functions $f$ which $\mathcal{C}^{1}$-converge to $1$ on $\Sigma^{\le \epsilon}$, each satisfying the assumptions of \S \ref{Sec:DetailedDynamics}, we can assume that for a given $L > 0$ all contractible Reeb orbits of action $\leq L$ are of the form $\orbit_{\zeta} = \orbit\times \zeta$ for some $\zeta \in \Crit(f)$ and $\gamma$ is a closed, contractible orbit of $R_{Y}$, the Reeb vector field of the contact form $\alpha_{Y}$ on $Y$.

\begin{notation}
For the remainder of this section, all Reeb orbits considered will be contractible in $\hypersurface = \Circle \times M$ with action bounded by some (fixed, but unspecified) $L > 0$.
\end{notation}

Let $CC(\divSet)$ be the algebra of good Reeb orbits in $\divSet$ of action $\leq L$ with contact homology differential $\partial_{\divSet}$. To verify that $CH(\R \times \Circle \times M) \neq 0$ in \S \ref{Sec:NonVanishing} using the Algebraic Giroux Criteria of \cite{Avdek:Hypersurface}, we study the augmentation $\aug_{W}: (CC(\divSet), \partial_{\divSet}) \rightarrow (\Q, \partial_{\Q}=0)$ determined by the filling $W$ in the case when $\Sigma$ is an annulus.

\subsection{Moduli spaces}\label{Sec:ModSpaceDefs}

We describe the moduli spaces relevant to the computation of the augmentation $\aug_{W}$ of $CC(\divSet)$. Throughout we write $\orbit$ for a single orbit in $Y$ and $\vec{\orbit} = (\orbit_{1}, \dots, \orbit_{k})$ for ordered collections of orbits in $Y$. For $\zeta \in \Crit(f)$, write $\orbit_{\zeta}$ for a single orbit in $\divSet$ of the form $\orbit \times \zeta$ and $\vec{\orbit}_{\zeta}$ for ordered collections or orbits in $\divSet$ of the form $(\orbit_{1}\times \zeta, \dots, \orbit_{k}\times \zeta)$. So in the latter case all of the $\orbit_{i} \times \zeta$ are contained in the same $Y \times \{\zeta\} \subset Y \times \Sigma^{\le \epsilon} \subset \divSet$. We allow these ordered collections to be empty.

Our moduli spaces will be denoted as follows:
\be
\item $\ModSpace_{\completion{V}}(\orbit)$ is the moduli space of $J_{V}$-holomorphic planes in $\completion{V}$, positively asymptotic to $\orbit$.
\item $\ModSpace_{Y}(\orbit, \vec{\orbit})$ is the moduli space of $J_{\divSet}$-holomorphic curves in $\completion{ 
Y}$, positively asymptotic to $\orbit$ and negatively asymptotic to the $\orbit_{i}$.
\item $\ModSpace_{\completion{W}}(\orbit_{\zeta})$ is the moduli space of $J_{W}$-holomorphic planes in $\completion{W}$, positively asymptotic to $\orbit_{\zeta}$.
\item $\ModSpace_{\divSet}(\orbit_{\zeta}, (\orbit_{1}\times \zeta_{1}, \cdots, \orbit_{k}\times \zeta_{k}))$ is the moduli space of $J_{\divSet}$-holomorphic curves in $\completion {\divSet}$, positively asymptotic to $\orbit \times \zeta$ and negatively asymptotic to the $\orbit_{i} \times \zeta_{i}$.
\ee

According to Lemma \ref{Lemma:CurveInLeaf} and the characterization of leaves of $\foliation_{\divSet}$ in \Cref{lemma:foliation_symp}, the $\ModSpace_{\divSet}$ with are empty unless they are of the form $\ModSpace_{\divSet}(\orbit_{\zeta}, \vec{\orbit}_{\zeta_{-}})$ with either
\be
\item $\vec{\orbit}_{\zeta_{-}}\neq\emptyset$ and $\zeta = \zeta_{-}$, whence the elements of $\ModSpace_{\divSet}(\orbit_{\zeta}, \vec{\orbit}_{\zeta})$ are mapped to $\R_{s} \times Y \times \{\zeta\}$,
\item $\vec{\orbit}_{\zeta_{-}}\neq \emptyset$, $\zeta = \zeta_{2}$, and $\zeta_{-}=\zeta_{1, k}$, here the elements of $\ModSpace_{\divSet}(\orbit_{\zeta_{2}}, \vec{\orbit}_{\zeta_{1, k}})$ are mapped into flow line leaves $\leaf_{\upsilon}$ associated to $\grad f$ flow lines $\upsilon$ in $\Sigma^{\le \epsilon}$ negatively asymptotic to one of the $\ind_{\Morse}=1$ critical points $\zeta_{1, k}$ of $f$,
\item $\vec{\orbit}_{\zeta_{-}}=\emptyset$, whence the moduli space consists of holomorphic planes.
\ee
In the first and second cases, let $\upsilon$ be a parameterized $\grad f$-flow flow positively asymptotic to $\zeta$ and negatively asymptotic to $\zeta_{-}$. Possibly after applying a $\R_{s}$ translation on the target, each $u \in \ModSpace_{\divSet}(\orbit_{\zeta}, \vec{\orbit}_{Y} \times \zeta_{-})$ is mapped to the image of the $(J_{Y}, J_{\divSet})$-holomorphic embedding $\leafInclusion_{\upsilon}$. Therefore the map
\begin{equation*}
\leafInclusion_{\upsilon, \ModSpace}: \ModSpace_{Y}(\orbit, \vec{\orbit}_) \rightarrow \ModSpace_{\divSet}(\orbit_{\zeta}, \vec{\orbit}_{\zeta}), \quad u \mapsto \leafInclusion_{\upsilon}\circ u
\end{equation*}
determines a homeomorphism on the $\R_{s}$-reduced spaces
\begin{equation*}
\leafInclusion_{\upsilon, \ModSpace/\R}: \ModSpace_{Y}(\orbit, \vec{\orbit})/\R_{s} \rightarrow \ModSpace_{\divSet}(\orbit_{\zeta}, \vec{\orbit}_{\zeta})/\R_{s}.
\end{equation*}
Observe that if $\upsilon$ and $\upsilon'$ are two $\grad f$ flow lines which differ by a reparameterization $\upsilon'(s)=\upsilon(s+ s_{0})$, then $\leafInclusion_{\upsilon, \ModSpace/\R} = \leafInclusion_{\upsilon', \ModSpace/\R}$. In the third case above (in which we are considering a $\ModSpace_{\divSet}$ moduli space of planes), the planes can be viewed as living in $\completion{W}$.

To compute $\aug_{W}(\gamma_{\zeta})$ we will count gluings of holomorphic buildings solving a perturbed $\delbar$ equation which are topologically planes in $\completion{W}$. Following \cite{BH, Pardon:Contact}, such buildings can be organized as elements of spaces of trees $\tree$, with moduli spaces assigned to vertices and gluing data (common Reeb orbit asymptotics) associated to edges. In order that these assignments will correspond to planes the following restrictions are imposed:
\be
\item Moduli spaces of the form $\ModSpace_{\completion{W}}(\orbit_{\zeta})$ and $\ModSpace_{\divSet}(\orbit_{\zeta}, \emptyset)$ are assigned to ``leaf vertices'', having no outgoing edges.
\item Moduli spaces of the form $\ModSpace_{\divSet}/\R_{s}$ with non-empty collections of negative asymptotics are assigned to each vertex with at least one outgoing edge. As previously mentioned, these are the images of the $\leafInclusion_{\upsilon, \ModSpace/\R}$ for a finite collection of $\grad f$ flow lines $\upsilon$.
\item Closed orbits assigned to edges must be contractible (since they are to be filled in by planes), and hence are of the form $\orbit_{\zeta} = \orbit \times \zeta$, for $\zeta \in \Crit{f}$
\ee
See \S \ref{Sec:TreeCharts} for further details, where moduli spaces will be replaced with thickened moduli spaces, solving perturbed $\delbar$ equations.

\subsection{Manifolds of maps}

Our convention is that moduli spaces are of parameterized maps with removable punctures defined as in \cite{Avdek:Hypersurface, BH}. This allows us to write
\begin{equation*}
\begin{gathered}
\ModSpace_{Y}(\orbit, \vec{\orbit}) \subset \Map_{\completion{Y}}(\orbit, \vec{\orbit}), \quad \ModSpace_{\completion{V}}(\orbit) \subset \Map_{\completion{V}}(\orbit)\\
\ModSpace_{\divSet}(\orbit_{\zeta}, \vec{\orbit}_{\zeta_{-}}) \subset \Map_{\completion{\divSet}}(\orbit_{\zeta}, \vec{\orbit}_{\zeta_{-}}), \quad \ModSpace_{\completion{W}}(\orbit) \subset \Map_{\completion{W}}(\orbit_{\zeta}),
\end{gathered}
\end{equation*}
where the $\Map_{X}$ spaces are Banach manifolds of maps of punctured Riemann surfaces $(S \subset \C, \domainJ)$ into a space $X$ with
\begin{equation*}
X=\completion{Y}, \completion{V}, \completion{\divSet}, \completion{W}
\end{equation*}
satisfying asymptotic convergence constraints along punctures. After deleting removable punctures, the domains of maps are stable (in the sense of hyperbolic geometry) and so have finite automorphism groups.

Over the manifolds of maps $\Map_{X}$ we have Banach bundles $\Omega^{0}$ and $\Omega^{0, 1}$ whose fibers at a map $u$ are (Sobolev complete) spaces of sections
\begin{equation*}
\Omega^{0}|_{u} =\Sec(u^{\ast}TX), \quad \Omega^{0, 1}|_{u} = \left\{ v \in \Sec((u^{\ast}TX) \otimes T^{\ast}S)\ : J_{X}v = -v\circ \domainJ \right\},
\end{equation*}
over the domain $(S, \domainJ)$ of $u$ respectively, satisfying decay conditions along cylindrical ends of $S$. We assume as in \cite{Avdek:Hypersurface} that our maps into $\completion{\divSet}$ and $\completion{W}$ take the form
\begin{equation}\label{Eq:PositiveEnds}
u: [0, \infty)_{p}\times \Circle_{q} \rightarrow [0, \infty)_{s} \times \Circle_{t} \times \disk_{y}^{2n-2}  \times \disk, \quad u(s,t) = (p + s_{0}, q, u_{Y}, u_{\disk})
\end{equation}
along positive ends of maps which are positively asymptotic to the $\orbit_{\zeta}$. Here, as in \S \ref{Sec:Asymptotics}, $\Circle \times \disk_{y}^{2n-2}$ is a neighborhood of a Reeb orbit in $Y$ and $\disk$ is a neighborhood of $\zeta \in \Crit(f)$. Details of the appropriate (weighted) Sobolev spaces are not relevant for the present arguments and can be found in \cite{BH}. The setup allows us to view $\delbar_{\domainJ, J_{X}}$ as a fiber-wise bundle morphism $\Omega^{0} \rightarrow \Omega^{0, 1}$ over $\Map_{X}$ whose zero set is the relevant moduli space of holomorphic maps.

\subsection{Spaces of perturbations}

Following \cite{BH}, the perturbations used to determine contact homology differentials and augmentations will be constructed using the eigendecompositions of asymptotic operators. The asymptotic operator for each orbit $\gamma_{\zeta} = \orbit \times \zeta$ has the form
\begin{equation*}
\AsymptoticOp_{\orbit_{\zeta}} = \AsymptoticOp_{\orbit} \oplus \AsymptoticOp_{\zeta}
\end{equation*}
where $\AsymptoticOp_{\orbit}$ is the asymptotic operator associated to the orbit $\orbit$ in $Y$ and $\AsymptoticOp_{\zeta}$ is is as described in Equation \eqref{Eq:NormalAsymptoticOp}. For each $\eigenbound \in \R_{>0}$, let $\eigenspace^{\leq \eigenbound}_{\orbit_{\zeta}}$ be the finite dimensional vector space spanned by eigenfunctions of $\AsymptoticOp_{\orbit_{\zeta}}$ whose eigenvalues are bounded in absolute value by $\eigenbound$. By the above direct sum splitting, it follows that
\begin{equation*}
\eigenspace^{\leq \eigenbound}_{\orbit_{\zeta}} = \eigenspace^{\leq \eigenbound}_{\orbit} \oplus \eigenspace^{\leq \eigenbound}_{\zeta}
\end{equation*}
where the $\eigenspace^{\leq \eigenbound}_{\orbit}$ and $\eigenspace^{\leq \eigenbound}_{\zeta}$ are defined analogously using the operators $\AsymptoticOp_{\orbit}$ and $\AsymptoticOp_{\zeta}$, respectively. For a tuple $\vec{\orbit}_{\zeta} = (\orbit_{1} \times \zeta, \dots, \orbit_{k}\times \zeta)$ of orbits in $\divSet$ write $\eigenspace^{\leq \eigenbound}_{\vec{\orbit}_{\zeta}} = \bigoplus_{i} \eigenspace^{\leq \eigenbound}_{\orbit_{i} \times \zeta}$ which likewise decomposes into $Y$ and $\zeta$ summands,
\begin{equation*}
\eigenspace^{\leq \eigenbound}_{\vec{\orbit}_{\zeta}} = \eigenspace^{\leq \eigenbound}_{\vec{\orbit}} \oplus \left(\eigenspace^{\leq \eigenbound}_{\zeta}\right)^{\oplus k}.
\end{equation*}

The $\eigenspace_{\ast}^{\leq \eigenbound}$ determine finite rank subbundles $\transSubbundle_{X}^{\leq \eigenvalue} \subset \Omega^{0, 1}$ over the manifolds of maps $\Map_{X}$ spanned by perturbations supported on cylindrical ends of the domains $(S, \domainJ)$ of our maps. For $X=Y, \completion{V}$, there is one perturbation supported on the positive cylindrical end of $S$ for each eigenfunction of $\AsymptoticOp_{\orbit}$ having eigenvalue $\eigenvalue \in (0, \eigenbound]$ and one perturbation supported on the $i$th negative end of $S$ for each eigenfunction of $\AsymptoticOp_{\orbit_{i}}$ having eigenvalue $\eigenbound \in [-\eigenbound, 0)$. For $X=\divSet, \completion{W}$, there is one perturbation supported on the positive cylindrical end of $S$ for each eigenfunction of $\AsymptoticOp_{\orbit_{\zeta}}$ having eigenvalue $\eigenvalue \in (0, \eigenbound]$ and one perturbation supported on the $i$th negative end of $S$ for each eigenfunction of $\AsymptoticOp_{\orbit_{i} \times \zeta}$ having eigenvalue $\eigenvalue \in [-\eigenbound, 0)$.

For the determination of an element of $\transSubbundle_{X}, X=\completion{Y}, \completion{W}$ from an eigenfunction we follow \cite[\S 9]{Avdek:Hypersurface}, working explicit formulas only the perturbations positive ends of curves. Let $B=B(p)$ be a function on $[0, \infty)$ for which $B(0)=1$ and $B(p)=1$ for $p\geq 1$. Let $u$ be a positive half-cylindrical end of a map asymptotic to some $\orbit_{\zeta}$ as in Equation \eqref{Eq:PositiveEnds}, with $\eigenfunction_{\orbit}$ (respectively, $\eigenfunction_{\zeta}$) an eigenfunction of $\AsymptoticOp_{\orbit}$ (respectively, $\AsymptoticOp_{\zeta}$) of eigenvalue $\eigenvalue_{\orbit}$ (respectively, $\eigenvalue_{\zeta}$). Our associated perturbations associated to such eigenfunctions at $u$ take the form
\begin{equation}\label{Eq:PertubationExplicit}
\begin{gathered}
u: [0,\infty)_{p} \times \Circle_{q} \rightarrow [0, \infty)_{s} \times \Circle_{t} \times \disk_{y}^{2n-2} \times \disk, \quad 0 < \eigenvalue_{\orbit}, \eigenvalue_{\zeta} \leq \eigenbound\\
\mu(\eigenfunction_{\orbit}) = \frac{\partial B}{\partial p}(0, 0, e^{-\eigenvalue_{\orbit} p}\eigenfunction_{\orbit}(q), 0)\otimes dp^{0, 1},\\
\mu(\eigenfunction_{\zeta}) = \frac{\partial B}{\partial p}(0, 0, 0, e^{-\eigenvalue_{\zeta} p}\eigenfunction_{\zeta}(q))\otimes dp^{0, 1}.
\end{gathered}
\end{equation}
So the supports of all perturbations are compact and contained in the half-cylindrical ends of their domains. Whenever a $\delbar u \in \transSubbundle_{X}^{\leq \eigenbound}$, then in the situation of Equation \eqref{Eq:PertubationExplicit} we have that along a positive half-cylindrical end,
\begin{equation}\label{Eq:PerturbedFourier}
\begin{gathered}
u = u_{\C} + u_{\mu}\\
u_{\C} = \left(p + s_{0}, q, \sum_{\eigenvalue_{\orbit}<0}e^{p\eigenvalue_{\orbit}}\sum_{\eigenfunction_{\orbit, i} \in \eigenspace_{\eigenvalue_{\orbit}}} a_{\C,i}\eigenfunction_{\orbit, i}(q), \sum_{\eigenvalue_{\zeta}<0}e^{p\eigenvalue_{\zeta}}\sum_{\eigenfunction_{\zeta, i} \in \eigenspace_{\eigenvalue_{\zeta}}} b_{\C,i}\eigenfunction_{\zeta, i}(q)\right),\\
u_{\mu} = (B-1)\left(0,0, \sum_{0<\eigenvalue_{\orbit}<\eigenbound}e^{-p\eigenvalue_{\orbit}}\sum_{\eigenfunction_{\orbit, i} \in \eigenspace_{\eigenvalue_{\orbit}}} a_{\mu,i}\eigenfunction_{\zeta, i}(q), \sum_{0<\eigenvalue_{\zeta}<\eigenbound}e^{-p\eigenvalue_{\zeta}}\sum_{\eigenfunction_{\orbit, i} \in \eigenspace_{\eigenvalue_{\zeta}}} b_{\mu,i}\eigenfunction_{\zeta, i}(q)\right)\\
\implies \delbar u = \sum a_{\mu, i}\mu(\eigenfunction_{\orbit, i}) + \sum b_{\mu, i}\mu(\eigenfunction_{\zeta, i})
\end{gathered}
\end{equation}
So $u_{\C}$ is holomorphic and $u_{\mu}$ completely determines $\delbar u$. This follows the pattern of \cite[\S 7.6]{Avdek:Hypersurface}.

Again, we have splittings
\begin{equation*}
\transSubbundle_{\completion{W}}^{\leq \eigenbound} = \transSubbundle_{\foliation}^{\leq \eigenbound} \oplus \transSubbundle_{\perp}^{\leq \eigenbound} \rightarrow \Map_{\completion{W}}, \quad \transSubbundle_{\divSet}^{\leq \eigenbound} = \transSubbundle_{\foliation}^{\leq \eigenbound} \oplus \transSubbundle_{\perp}^{\leq \eigenbound} \rightarrow \Map_{\completion{\divSet}}
\end{equation*}
where the $\transSubbundle_{\foliation_{W}}^{\leq \eigenbound}, \transSubbundle_{\foliation_{\divSet}}^{\leq \eigenbound}$ are spanned by perturbations $\mu(\eigenfunction_{\orbit})$ associated to the $\AsymptoticOp_{\orbit}$ and the $\transSubbundle_{\perp}^{\leq \eigenbound}$ are spanned by perturbations $\mu(\eigenfunction_{\zeta})$ associated to the $\AsymptoticOp_{\zeta}$. The important property from Equation \eqref{Eq:PertubationExplicit} is that in the notation of \S \ref{Sec:Fredholm}, where the $\Omega^{0, 1}_{\foliation}$ and $\Omega^{0, 1}_{\perp}$ are defined,
\begin{equation*}
\transSubbundle_{\foliation}^{\leq \eigenbound} \subset \Omega^{0, 1}_{\foliation}, \quad \transSubbundle_{\perp}^{\leq \eigenbound} \subset \Omega^{0, 1}_{\perp}.
\end{equation*}
This follows from the fact that along $[0, \infty) \times \Circle \times \disk^{2n-2}_{y} \times \disk$ is given by
\begin{equation*}
T\foliation = \left\langle (1,0,0,-\grad f), (0,1,0,0), T\disk^{2n-2}_{y} \right\rangle.
\end{equation*}

To construct transverse subbundles for maps into the spaces $X=\completion{Y}, \completion{V}$, we follow the above formulas ignoring $u_{\disk}$ summands of maps $u$ as above as well as the $\eigenfunction_{\zeta}$. This yields $\transSubbundle_{X} = \langle \mu(\eigenfunction_{\orbit})\ : \eigenfunction_{\orbit} \in \eigenspace_{\eigenvalue_{\orbit}}, |\eigenvalue_{\orbit}| \leq \eigenbound \rangle$. If we have a map $u$ into either of these spaces, then along a positive cylindrical end $u=(p+s_{0}, t, u_{y}): [0, \infty) \times \Circle \rightarrow [0, \infty) \times\Circle \times \disk_{y}$ with $\Circle \times \disk_{y}$ a neighborhood of the orbit in $Y$, then applying some leaf inclusion $\leafInclusion_{\leaf}$, we will have $\leafInclusion_{\leaf}\circ u = (p+s_{0}, t, u_{y}, \upsilon(s))$ where $\upsilon$ is a $\grad f$ solution. So the following lemma is immediate from Equation \eqref{Eq:PerturbedFourier}.

\begin{lemma}\label{Lemma:LeafInclusionDbar}
Let $u$ be a map into $\R_{s} \times Y$ (respectively, $\completion{V}$) satisfying $\delbar u = \sum a_{\mu, i}\mu(\eigenfunction_{\orbit, i})$ in $\transSubbundle^{\leq a}_{\R_{s} \times Y}$ (respectively, $\transSubbundle^{\leq \eigenbound}_{\completion{V}}$). If we apply a leaf inclusion map $\leafInclusion$ to $u$ then $\leafInclusion\circ u$ will satisfy $\delbar \leafInclusion u = \sum a_{\mu, i}\mu(\eigenfunction_{\orbit, i})$ in $\transSubbundle^{\leq a}_{\R_{s} \times \divSet}$ (respectively, $\transSubbundle^{\leq \eigenbound}_{\completion{W}}$).
\end{lemma}

\subsection{Transversality over compact subsets}

A finite rank subbundle $\transSubbundle \subset \Omega^{0, 1}$ defined over a subset $U \subset \Map_{X}$ is a \emph{transverse subbundle} if for all $u \in U$
\begin{equation*}
\im \Dlinearized_{u} + \transSubbundle = \Omega^{0, 1}|_{u}.
\end{equation*}
Since transversality is an open condition, if $\transSubbundle$ is a transverse subbundle over a compact $U \subset \Map_{X}$ and $\transSubbundle$ is defined over an open neighborhood of $U$ in $\Map_{X}$, then there is an open $U' \subset \Map_{X}$ containing $U$ such that $\transSubbundle|_{U'}$ is a transverse subbundle. Following the Bao-Honda perturbation scheme, we must select compact subsets $\compactSubset_{\completion{W}}$ and $\compactSubset_{\divSet}/\R$ of our $\ModSpace_{\completion{W}}$ and $\ModSpace_{\divSet}/\R$ moduli spaces over which the spaces of perturbations described in the previous subsection are transverse.

Following the already-establishing pattern, the Kuranishi data for the higher-dimensional spaces ($\completion{W}, \completion{\divSet}$) will be determined by Kuranishi data associated to the lower-dimensional spaces ($\completion{V}, \completion{Y}$). So suppose that such $\compactSubset_{Y}/\R \subset \ModSpace_{Y}/\R$ and $\compactSubset_{\completion{V}} \subset \ModSpace_{\completion{V}}$ have been selected for each moduli space whose positive asymptotic $\orbit$ has action $\leq L$. We also assume that $\eigenbound > 0$ chosen so that over each such compact subset the subbundles $\transSubbundle_{\completion{Y}}^{\leq \eigenbound} \rightarrow \compactSubset_{Y}$ and $\transSubbundle_{\completion{V}}^{\leq \eigenbound} \rightarrow \compactSubset_{\completion{V}}$ are transverse.

\subsubsection{When $\vec{\orbit}_{\zeta} \neq \emptyset$}

We first determine the $\compactSubset_{\divSet}(\orbit_{\zeta}, \vec{\orbit}_{\zeta^{-}} \neq \emptyset)/\R$ by declaring that
\begin{equation*}
\compactSubset_{\divSet}(\orbit_{\zeta}, \vec{\orbit}_{\zeta^{-}})/\R = \leafInclusion_{\upsilon, \ModSpace/\R}\left(\compactSubset_{Y}(\orbit, \vec{\orbit})/\R\right).
\end{equation*}
Here, as always, the $\upsilon$ are $\grad f$ flow lines in $\Sigma^{\le \epsilon}$ (considered modulo $\R_{s}$-translation) which are negatively asymptotic to $\zeta^{-}$ and positively asymptotic to $\zeta$.

By continuity of $\leafInclusion_{\upsilon, \ModSpace/\R}$ the $\compactSubset_{\divSet}(\orbit_{\zeta}, \vec{\orbit}_{\zeta^{-}})/\R$ are compact. Since for each holomorphic $u \in \Map_{\completion{Y}}$,
\begin{equation}\label{Eq:LeafTangentPushforward}
(\leafInclusion_{\upsilon}u)^{\ast}T\foliation_{Y} = (\leafInclusion_{\upsilon}u)^{\ast}T(\im \leafInclusion_{\upsilon})= T\leafInclusion_{\upsilon}(u^{\ast}T(\completion{Y})),
\end{equation}
the transversality condition $\im\Dlinearized_{u} + \transSubbundle_{Y}^{\leq \eigenbound} = \Omega^{0, 1}(u^{\ast}(\R_{s}\times Y))$ implies the transversality condition
\begin{equation*}
\im\Dlinearized_{\foliation} + \transSubbundle_{\foliation_{Y}}^{\leq \eigenbound} = \Omega^{0, 1}(u^{\ast}(T\foliation)).
\end{equation*}
Here $\Dlinearized_{\foliation}$ is as in Equation \eqref{Eq:DBlockMatrix}.

From the compactness of $\compactSubset_{\divSet}(\orbit_{\zeta}, \vec{\orbit}_{\zeta^{-}})/\R$, $\R_{s}$ invariance, and the fact that there are only finitely many such compact spaces to consider (as per our action bound), it follows that we can guarantee transversality of $\transSubbundle_{\divSet}^{\leq \eigenbound} \rightarrow \compactSubset_{\divSet}(\orbit_{\zeta}, \vec{\orbit}_{\zeta^{-}})$ over all such relevant compact sets, possibly after increasing $\eigenbound$.

\subsubsection{When $\vec{\orbit}_{\zeta} = \emptyset$}

Now we concern ourselves with compact subsets of the $\ModSpace_{\completion{W}}(\orbit_{\zeta})$ and $\ModSpace_{\divSet}(\orbit_{\zeta}, \emptyset)/\R_{s}$, whose elements map to leaves of $\foliation_{W}$ and $\foliation_{\divSet}$, respectively.

Suppose that we have selected compact subsets $\compactSubset_{\completion{V}} \subset \ModSpace_{\completion{V}}$ for each such relevant $\ModSpace_{\completion{V}} = \ModSpace_{\completion{V}}(\orbit)$. For each $\zeta \in \Crit(f) \subset \Sigma^{\le \epsilon}$, let $U_{\zeta} \subset \completion{\Sigma}$ be the subset of $z \in \completion{\Sigma}$ for which $\im\leafInclusion_{z} \subset \completion{W}$ is positively asymptotic to $Y \times \{ \zeta\}$ as described in Lemma \ref{Lemma:LeafAsymptotics}.

Since holomorphic curves in $\completion{W}$ are contained in leaves of $\foliation_{W}$, it follows that every $u \in \ModSpace_{\completion{W}}(\orbit_{\zeta})$ has can be uniquely expressed as $u = \leafInclusion_{z}\circ u_{\completion{V}}$ for some $u_{\completion{V}} \in \ModSpace_{\completion{V}}(\orbit)$. Therefore we have homeomorphisms
\begin{equation*}
\leafInclusion_{\zeta, \ModSpace}: U_{\zeta} \times \ModSpace_{\completion{V}}(\orbit) \rightarrow \ModSpace_{\completion{W}}(\orbit_{\zeta}), \quad (z, u_{\completion{V}}) \mapsto \leafInclusion_{z}\circ u_{\completion{V}}.
\end{equation*}
For a compact $\compactSubset_{\completion{V}} \subset \ModSpace_{\completion{V}}$ over which $\transSubbundle_{V}^{\leq \eigenbound}$ is a transverse subbundle, a slight notational modification of Equation \eqref{Eq:LeafTangentPushforward} yields
\begin{equation*}
\im \Dlinearized_{\foliation} + \transSubbundle_{\foliation_{W}}^{\leq \eigenbound} = \Omega^{0, 1}(u^{\ast}T\foliation_{W}), \quad u \in \leafInclusion_{\zeta, \ModSpace}\left(U_{\zeta} \times \compactSubset_{\completion{V}}\right).
\end{equation*}
Since the normal operator $\Dlinearized_{\perp}$ is always surjective -- as described in Equation \eqref{Eq:NormalPlaneAutomaticTransversality} -- it follows that
\begin{equation*}
\im \Dlinearized_{u} + \transSubbundle_{\completion{W}}^{\leq \eigenbound} = \Omega^{0, 1}, \quad u \in \leafInclusion_{\zeta, \ModSpace}\left(U_{\zeta} \times \compactSubset_{\completion{V}}\right).
\end{equation*}

Due to the non-compactness of the $U_{\zeta}$, the $\leafInclusion_{\zeta, \ModSpace}\left(U_{\zeta} \times \compactSubset_{\completion{V}}\right)$ are non-compact. To complete our construction we will consider special subsets of the $U_{\zeta}$. We can write $U_{\zeta}$ as a union of open subsets $U_{\zeta} =  U_{\zeta, \divSet} \subset U_{\zeta, W}$ defined as follows:
\be
\item For $z \in U_{\zeta, W}$ each $\leafInclusion_{z}\completion{V} \subset \overline{W}$ touches the complement of $(0, \infty) \times \divSet$ in $\completion{W}$.
\item For $z \in U_{\zeta, \divSet}$ each $\leafInclusion_{z}\completion{V}$ is contained in the interior $(0, \infty) \times \divSet$ of the half cylindrical end of $\completion{W}$.
\ee

Observe that $U_{\zeta, W}$ has compact closure in $\completion{\Sigma}$ while $U_{\zeta, \divSet}$ does not. Choosing compact subsets $\compactSubset_{\zeta, W} \subset U_{\zeta, W}$ for each $\zeta \in \Crit(f)$ we define
\begin{equation*}
\compactSubset_{\completion{W}}(\orbit_{\zeta}) = \leafInclusion_{\zeta, \ModSpace}\left(\compactSubset_{\zeta, W} \times \compactSubset_{\completion{V}}(\orbit)\right) \subset \ModSpace_{\completion{W}}(\orbit_{\zeta}).
\end{equation*}

The $U_{\zeta, \divSet}$ have a $s\in (0, \infty)$ action given by translating holomorphic curves upward in the symplectization. The action is free and we choose compact $\compactSubset_{\zeta, \divSet}/\R \subset U_{\zeta, \divSet}/\R$. Holomorphic curves in $(0, \infty) \times \divSet$ as above -- where $J_{W}$ is $s$-invariant -- can be viewed as mapping in the symplecticization of $\divSet$ and therefore being elements of the $\ModSpace_{M}(\orbit_{\zeta})$ moduli spaces. We take our compact subsets of these (reduced) moduli spaces to be
\begin{equation*}
\compactSubset_{\divSet}(\orbit_{\zeta}, \emptyset)/\R = \leafInclusion_{\zeta, \ModSpace}\left( \compactSubset_{\zeta, \divSet}/\R \times \compactSubset_{\completion{V}}(\orbit)\right) \subset \ModSpace_{\divSet}(\orbit_{\zeta})/\R.
\end{equation*}

\subsection{Thickened moduli spaces}

Now we describe thickenings of our compact sets.

\begin{notation}
Having fixed a bound $\eigenbound$ on the eigenvalues of our asymptotic operators in the previous subsection, we drop the $\leq \eigenbound$ superscripts from our transverse subbundles moving forward.
\end{notation}

Choose small neighborhoods
\begin{equation*}
\NcompactSubset_{\completion{V}} \supset \compactSubset_{\completion{V}}, \quad \NcompactSubset_{\completion{W}} \supset \compactSubset_{\completion{W}} \quad \NcompactSubset_{Y} \supset \compactSubset_{Y}, \quad \NcompactSubset_{\divSet} \supset \compactSubset_{\divSet}
\end{equation*}
in the corresponding manifolds of maps. We require that the $\NcompactSubset_{Y}$ and $\NcompactSubset_{\divSet}$ are $\R_{s}$-invariant. Over the $\NcompactSubset$ our spaces of perturbations determine finite rank bundles $\transSubbundle_{\ast} \rightarrow \NcompactSubset_{\ast}$ and we define
\begin{equation}\label{Eq:ModSpaceThiccDef}
\begin{gathered}
\ModSpaceThicc_{\completion{V}} \subset \NcompactSubset_{\completion{V}}, \quad \ModSpaceThicc_{\completion{W}} \subset \NcompactSubset_{\completion{W}}, \quad \ModSpaceThicc_{Y} \subset \NcompactSubset_{Y}, \quad \ModSpaceThicc_{\divSet} \subset \NcompactSubset_{\divSet},\\
\ModSpaceThicc_{\ast} = \{ u \in \NcompactSubset_{\ast}\ :\ \delbar u \in \transSubbundle_{\ast}, \quad \norm{\delbar u} \leq \rho\}.
\end{gathered}
\end{equation}
Here $\rho$ is an arbitrarily small positive constant and the $\norm{\delbar u}$ can be defined using a metric on $\transSubbundle_{\ast}$ which we require to be $\R_{s}$ invariant in the symplectization cases. The $\ModSpaceThicc_{\ast}$ are known as interior Kuranishi charts in the language of \cite{BH}. By taking the $\rho$ to be sufficiently small, the $(\transSubbundle_{\ast} \subset \Omega^{0, 1}) \rightarrow \ModSpaceThicc_{\ast}$ are transverse subbundles by the openness of the transversality condition combined with transversality along the $\compactSubset$. Consequently, the $\ModSpaceThicc_{\ast}$ are smooth manifolds when they are constructed using parameterized maps. We can alternatively view the $\transSubbundle_{\ast} \rightarrow \ModSpaceThicc_{\ast}$ as orbibundles over smooth orbifolds by taking into account isotropy groups determined by automorphisms of the domains of the $u$. The distinction will not matter for the following arguments.

Lemma \ref{Lemma:LeafInclusionDbar} tells us that if we apply the leaf inclusion maps to the $\ModSpaceThicc_{Y}/\R$ and $\ModSpaceThicc_{\completion{V}}$ we obtain embeddings
\begin{equation*}
\begin{gathered}
\leafInclusion_{\upsilon, \ModSpaceThicc}: \ModSpaceThicc_{Y}(\orbit, \vec{\orbit})/\R \rightarrow \ModSpaceThicc_{\divSet}(\orbit_{\zeta}, \vec{\orbit}_{\zeta^{-}})/\R, \quad u_{Y} \mapsto \leafInclusion_{\upsilon}\circ u_{Y},\\
\leafInclusion_{\zeta, \ModSpaceThicc}: \compactSubset_{\zeta, \divSet}/\R \times \compactSubset_{\completion{V}}(\orbit) \rightarrow \ModSpaceThicc_{\divSet}(\orbit_{\zeta})/\R_{s}, \quad (z,u_{\completion{V}}) \mapsto \leafInclusion_{z}\circ u_{\completion{V}},\\
\leafInclusion_{\zeta, \ModSpaceThicc}: \compactSubset_{\zeta, W} \times \compactSubset_{\completion{V}}(\orbit) \rightarrow \compactSubset_{\completion{W}}(\orbit_{\zeta}), \quad (z,u_{\completion{V}}) \mapsto \leafInclusion_{z}\circ u_{\completion{V}}.
\end{gathered}
\end{equation*}
The first property of the following lemma is established by the definition of the leaf inclusion maps and the second has just been established.

\begin{lemma}\label{Lemma:PerturbedInLeaf}
The $\leafInclusion_{\upsilon, \ModSpaceThicc}$ and $\leafInclusion_{\zeta, \ModSpaceThicc}$ are such that if $u$ is a map to $\completion{\divSet}$ (or $\completion{W}$) which is the image of a $\leafInclusion_{\ast, \ModSpaceThicc}$, then
\be
\item the image of $u$ is contained in a leaf of $\foliation_{\divSet}$ (respectively, $\foliation_{W}$) and
\item $\delbar u$ lies in $\transSubbundle_{Y}$ (respectively, $\transSubbundle_{\completion{W}}$).
\ee
Moreover, by choosing the $\NcompactSubset_{\completion{W}}$ to be sufficiently small, we can guarantee that if $u \in \ModSpaceThicc_{\completion{W}}$ satisfies $\delbar u \in \transSubbundle_{\foliation_{W}}$, then $u$ is contained in a leaf of $\foliation_{W}$ and so is in the image of $\leafInclusion_{\zeta, \ModSpaceThicc}$.
\end{lemma}

\begin{proof}
We need only to work out the last statement. Since $\transSubbundle_{\completion{V}}$ is a transverse subbundle and $\leafInclusion_{\zeta, \ModSpaceThicc}$ is an embedding we see that the dimension of its image is
\begin{equation*}
\dim \compactSubset_{\zeta, W} + \dim \compactSubset_{\completion{V}} = \ind_{\Morse}(\zeta) + \ind(u_{\completion{V}}) + \rank \transSubbundle_{\completion{V}}
\end{equation*}
for $u_{\completion{V}} \in \ModSpaceThicc_{\completion{V}}$. Since $\transSubbundle_{\completion{V}}$ and $\transSubbundle_{\foliation_{W}}$ are both determined by the $\eigenvalue \leq \eigenbound$ eigenspaces of the asympotic operator $\AsymptoticOp_{\orbit}$, we have that their ranks agree.

Recall that $\transSubbundle_{\foliation_{W}}$ is a transverse subbundle over $\compactSubset_{\completion{W}}$ due to the automatic transversality of the ``normal'' linearized operator $\Dlinearized_{\perp}$ for planes. Hence
\begin{equation*}
\ModSpaceThicc_{\completion{W}, \foliation_{W}} = \{ u \in \ModSpaceThicc_{\completion{W}}\ :\ \delbar u \in \transSubbundle_{\foliation_{W}}\}
\end{equation*}
is a smooth manifold of dimension
\begin{equation*}
\dim \ModSpaceThicc_{\completion{W}, \foliation_{W}} = \ind(u) + \rank \transSubbundle_{\foliation} = \ind_{\Morse}(\zeta) + \ind \Dlinearized_{\foliation} + \rank \transSubbundle_{\foliation}.
\end{equation*}
Since the $\transSubbundle_{\completion{V}}$, $\transSubbundle_{Y}$, and $\transSubbundle_{\foliation}$ have the same rank, this dimension count exactly matches $\dim \im \leafInclusion_{\zeta, \ModSpaceThicc}$ and so we can ensure that $\leafInclusion_{\zeta, \ModSpaceThicc}$ is a homeomorphism by making $\NcompactSubset_{\completion{W}}$ small.
\end{proof}

\subsection{Multisections over interior charts}\label{Sec:MultisectionsInterior}

Over each $\ModSpaceThicc_{\ast}$, $u \mapsto \delbar u$ defines a section of $\transSubbundle_{\ast} \rightarrow \ModSpaceThicc_{\ast}$, which is not necessarily transversely cut out. To correct this, we choose $\mathcal{C}^{1}$-small multisections $\multisec: \ModSpaceThicc_{\ast} \rightarrow \transSubbundle$ for which the solution spaces $\{ u \in \ModSpaceThicc\ :\ \delbar u = \multisec\}$ are transversely cut out. All sections are required to vanish along the ``vertical boundaries'' $\{ u \in \ModSpaceThicc\ : \ \norm{\delbar u} = \rho \}$ \cite{BH}.

We choose our multisections for the $\ModSpaceThicc_{Y}, \ModSpaceThicc_{\completion{W}}$ by assuming that multisections $\multisec_{Y} \in \Sec(\ModSpaceThicc_{Y}, \transSubbundle_{Y})$ and $\multisec_{\completion{V}} \in \Sec(\ModSpaceThicc_{\completion{V}}, \transSubbundle_{\completion{V}})$ have already been selected so that transversality holds. We'll use the $\leafInclusion_{\zeta, \ModSpaceThicc}$ to push forward these sections over their images using the fact that for each map $u_{\completion{V}}$ into $\completion{V}$ and $z \in \completion{\Sigma}$, we have natural isomorphisms
\begin{equation*}
(\leafInclusion_{z}\circ u_{\completion{V}})^{\ast}\transSubbundle_{\foliation_{W}} = \transSubbundle_{\completion{V}}.
\end{equation*}
Define $\multisec_{\completion{W}}$ along the image of $\leafInclusion_{\zeta, \ModSpaceThicc}$ as
\begin{equation*}
\multisec_{\completion{W}}\left(\leafInclusion_{z}(u_{\completion{V}})\right) = \left(\leafInclusion_{z}(u_{\completion{V}})\right)^{\ast}\multisec_{\completion{V}}, \quad \multisec_{\completion{W}} \in \Sec(\im\leafInclusion_{\zeta, \ModSpaceThicc} \subset \ModSpaceThicc_{\completion{W}}, \transSubbundle_{\foliation_{W}}) \subset \Sec(\im\leafInclusion_{\zeta, \ModSpaceThicc}, \transSubbundle_{\completion{W}}).
\end{equation*}
We emphasize that $\multisec_{\completion{W}}$ has no $\transSubbundle_{\perp}$ summand and are independent of $z \in \compactSubset_{\completion{W}}$. Because the $\multisec_{\completion{V}}$ are transversely cut out, so are the $\multisec_{\completion{W}}$ along the image of $\leafInclusion_{\zeta, \ModSpaceThicc}$. We therefore can extend these multisections over the remainder of $\ModSpaceThicc_{\completion{W}}$ arbitrarily, subject to the requirement of transversality and vanishing along vertical boundaries.

Our multisections over the $\ModSpaceThicc_{\divSet}(\orbit_{\zeta}, \emptyset)$ are constructed with a slight modification, required by lack of transversality. In the case of the $\ModSpaceThicc_{\divSet}(\orbit_{\zeta}, \vec{\orbit}_{\zeta^{-}})$ we define $\multisec_{\divSet}$ over the image of the $\leafInclusion_{\upsilon, \ModSpaceThicc}$ so that for each $\grad f$ flow line $\upsilon$ and map $u_{Y}$ into $\R_{s} \times Y$ we have
\begin{equation*}
\begin{gathered}
\multisec_{\divSet} = \multisec_{\divSet, \foliation}\oplus \multisec_{\divSet, \perp}, \quad \multisec_{\divSet, \foliation}\left(\leafInclusion_{\upsilon}(u_{Y})\right) = \left(\leafInclusion_{\upsilon}(u_{Y})\right)^{\ast}\multisec_{Y}, \\
\multisec_{Y, \foliation} \in \Sec(\im\leafInclusion_{\upsilon, \ModSpaceThicc} \subset \ModSpaceThicc_{\divSet}, \transSubbundle_{\foliation_{\divSet}}), \quad \multisec_{Y, \perp} \in \Sec(\im\leafInclusion_{\upsilon, \ModSpaceThicc} \subset \ModSpaceThicc_{\divSet}, \transSubbundle_{\perp})
\end{gathered}
\end{equation*}
The transversality of $\multisec_{Y}$ guarantees that for $u$ in the image of $\leafInclusion_{\upsilon}$ that if $\delbar u = \multisec_{\divSet}$ than we have transversality for the $\Omega^{0,1}(u^{\ast}T\foliation_{\divSet})$ portion of $\Omega^{0, 1}$. However, since $\Dlinearized_{\perp}$ is not surjective in general (as a consequence of Equation \eqref{Eq:NormalNotPlaneIndex}), the $\multisec_{\divSet, \perp}$ cannot be $0$ in general, as was the case with thickened moduli spaces of holomorphic planes. As in the previous case, we then extend the $\multisec_{\divSet}$ to the remainder of $\ModSpaceThicc_{\divSet}$ subject to the constraints of transversality and vanishing along vertical boundaries.

\subsection{Extension over tree charts}\label{Sec:TreeCharts}

Now we upgrade the construction of the previous sections to more general Kuranishi charts indexed by contact homology trees $\tree$ with vertices $\Vertex_{\tree} = \{\vertex_{i}\}$ and $\Edge_{\tree} = \{ \edge_{j} \}$.

We first provide an overview of the general construction \cite{BH}, without specifying that we are working within $\completion{V}$, $\completion{Y}$, $\completion{W}$ or $\completion{\divSet}$. For each such $\tree$ we have interior charts $\transSubbundle^{i} \rightarrow \ModSpaceThicc^{i}$ associated to each $\vertex_{i}$ with edges $\edge_{j}$ connecting vertices which have common asymptotics. This data comes with a gluing map
\begin{equation*}
\Glue^{\tree}: \left(\prod \ModSpaceThicc^{i}\right) \times [C, \infty)_{\necklength_{j}}^{\# \Edge_{\tree}} \rightarrow \Map
\end{equation*}
into a manifold of maps for some $C \gg 0$. The $\necklength_{j}$ are \emph{neck length parameters}. Over the image we have a transverse subbundle
\begin{equation*}
\transSubbundle^{\tree} \rightarrow \im \Glue^{\tree}, \quad \Glue^{\ast}\transSubbundle^{\tree} = \boxtimes_{\vertex_{i}} \transSubbundle^{i}.
\end{equation*}
Multisections $\multisec^{\tree}$ over $\im\Glue^{\tree}$ are determined by a melding construction -- see \cite{BH} or the simplification in \cite{Avdek:Hypersurface} -- with the property that when the $\necklength_{j}$ are all sufficiently large, then $\Glue^{\ast}\multisec^{\tree}$ is $\mathcal{C}^{1}$ close to $\boxtimes \multisec^{i}$ where $\multisec^{i} \in \Sec(\ModSpaceThicc^{i}, \transSubbundle^{i})$ are the multisections chosen over the interior Kuranishi charts. The images of the $\Glue^{\tree}$ can overlap when neck length parameters are small, requiring additional compatibility conditions for the $\multisec^{\tree}$ so that the $\multisec^{\tree}$ patch together to form a coherently defined section of $\transSubbundle^{\tree}$ over the union of the images of all $\Glue^{\tree}$.

Now let's specify that we are applying the above setup to study perturbed holomorphic planes in $\completion{V}$. Here we write $\tree_{\completion{V}}$ for our trees. In this case each $\ModSpaceThicc^{i}$ will be either a $\ModSpaceThicc_{\completion{V}}$ or a $\ModSpaceThicc_{Y}/\R_{s}$ with the gluing maps $\Glue_{\completion{V}}^{\tree_{\completion{V}}}$ having domain $\dom \Glue_{\completion{V}}^{\tree_{\completion{V}}}$ and target $\Map_{\completion{V}}$, a manifold of maps of planes into $\completion{V}$ which are positively asymptotic to some closed orbit $\orbit$ of $R_{Y}$. The union of the solutions sets $\{ \delbar u = \multisec^{\tree_{\completion{V}}}\} \subset \im \Glue^{\tree}$ over all $\tree$ then form a branched, weighted orbifold with corners, admitting an orientation determined by some choices of framing data associated to the orbits of $R_{Y}$ having action $\leq L$.

Let's instead specify that we are applying the above setup to study perturbed holomorphic planes in $\completion{W}$. Here our trees will be denoted $\tree_{\completion{W}}$ and each $\ModSpaceThicc^{i}$ will be either a $\ModSpaceThicc_{\completion{W}}$ or a $\ModSpaceThicc_{\divSet}/\R_{s}$ with the gluing maps $\Glue_{\completion{W}}^{\tree_{\completion{W}}}$ having target $\Map_{\completion{W}}$, a manifold of maps of planes into $\completion{W}$ which are positively asymptotic to some closed orbit $\orbit_{\zeta}$ of $R_{\divSet}$. On the image of each $\Glue_{\completion{W}}^{\tree_{\completion{W}}}$ our transverse subbundle splits as
\begin{equation*}
\transSubbundle^{\tree_{\completion{W}}}_{\completion{W}} \rightarrow \im \Glue_{\completion{W}}^{\tree}, \quad \Glue^{\ast}\transSubbundle^{\tree_{\completion{W}}} = \Glue^{\ast}\left(\transSubbundle^{\tree}_{\foliation} \oplus \transSubbundle^{\tree_{\completion{W}}}_{\perp}\right), \quad \transSubbundle^{\tree_{\completion{W}}}_{\foliation} = \boxtimes_{\vertex_{i}} \transSubbundle_{\foliation}^{i}, \quad \transSubbundle^{\tree_{\completion{W}}}_{\perp} = \boxtimes_{\vertex_{i}} \transSubbundle_{\perp}^{i}.
\end{equation*}
Here $\transSubbundle_{\foliation}^{i}$ is a $\transSubbundle_{\foliation_{W}}$ if $\ModSpaceThicc^{i}$ is a $\ModSpaceThicc_{\completion{W}}$ and is a $\transSubbundle_{\foliation_{\divSet}}$ if $\ModSpaceThicc^{i}$ is a $\ModSpaceThicc_{\divSet}/\R_{s}$.

We briefly provide an explicit description of the $\transSubbundle^{\tree_{\completion{W}}}_{\completion{W}}$ subbundles for the purpose of demonstrating that our earlier results for interior Kuranishi charts carries over to the more general case of charts associated to trees. As in the case of (positive) half-cylinders described above, consider maps of annuli
\begin{equation*}
u: [0, \necklength]_{p} \times \Circle_{q} \rightarrow [0, \infty)_{s} \times \Circle_{t} \times \disk^{2n-2}_{y} \times \disk, \quad u(p, q) = (p + s_{0}, t, u_{y}, u_{\disk})
\end{equation*}
on which the target has the model $J$ as described in Equations \eqref{Eq:PositiveEnds}. For maps $u$ which contain such annuli, perturbations are associated to eigenvalues of $\AsymptoticOp_{\orbit_{\zeta}} = \AsymptoticOp_{\orbit} \oplus \AsymptoticOp_{\zeta}$ via the following formulas as a direct analogy of Equation \eqref{Eq:PertubationExplicit}:
\begin{equation}\label{Eq:PertubationExplicitAnnulus}
\begin{gathered}
0 < \eigenvalue_{\orbit}, \eigenvalue_{\zeta} \leq \eigenbound \implies \begin{cases} \mu(\eigenfunction_{\orbit}) = \frac{\partial B}{\partial p}(0, 0, e^{-\eigenvalue_{\orbit} p}\eigenfunction_{\orbit}(q), 0)\otimes dp^{0, 1},\\
\mu(\eigenfunction_{\zeta}) = \frac{\partial B}{\partial p}(0, 0, 0, e^{-\eigenvalue_{\orbit} p}\eigenfunction_{\orbit}(q))\otimes dp^{0, 1},
\end{cases}\\
-\eigenbound < \eigenvalue_{\orbit}, \eigenvalue_{\zeta} < 0 \implies \begin{cases} \mu(\eigenfunction_{\orbit}) = \frac{\partial B^{+}}{\partial p}(0, 0, e^{-\eigenvalue_{\orbit} p}\eigenfunction_{\orbit}(q), 0)\otimes dp^{0, 1},\\
\mu(\eigenfunction_{\zeta}) = \frac{\partial B^{+}}{\partial p}(0, 0, 0, e^{-\eigenvalue_{\orbit} p}\eigenfunction_{\orbit}(q))\otimes dp^{0, 1},
\end{cases}\\
B^{+}(p) = B(p-\necklength).
\end{gathered}
\end{equation}
So the $0 < \eigenvalue_{\orbit}, \eigenvalue_{\zeta} \leq \eigenbound$ perturbations correspond to those of Equation \eqref{Eq:PertubationExplicit} supported on positive half-cylindrical ends of maps. The $-\eigenbound < \eigenvalue_{\orbit}, \eigenvalue_{\zeta} < 0$ perturbations correspond to perturbations associated to negative half-cylindrical ends of maps. When the gluing operation is performed, such perturbations span a subspace of $\transSubbundle^{\tree_{\completion{W}}}_{\completion{W}}$ associated to a gluing neck of neck-length $\necklength$. This agrees with \cite[\S 7]{Avdek:Hypersurface} up to some changes in notation, and from which an analogue of Equation \eqref{Eq:PerturbedFourier} may be derived. Replacing $\completion{W}$ with $\completion{V}$ then amounts to throwing out the $u_{\disk}, \mu(\eta_{\zeta})$ terms above.

Fix some $\zeta \in \Crit(f)$ and a $z \in \compactSubset_{\zeta, W} \subset \completion{\Sigma}$ so that the leaf $\leafInclusion_{z}(\completion{V}) \subset \completion{W}$ is positively asymptotic to $Y \times \zeta \subset \divSet$. Now we consider the composition
\begin{equation*}
\leafInclusion_{z}\circ \Glue^{\tree_{\completion{V}}}_{\completion{V}}: \dom \Glue^{\tree_{\completion{V}}}_{\completion{V}} \rightarrow \Map_{\completion{W}}.
\end{equation*}
The analogue of Lemma \ref{Lemma:PerturbedInLeaf}, whose proof is identical (using the above formulas in place of Equations \eqref{Eq:PertubationExplicit} and \eqref{Eq:PerturbedFourier}), is as follows.

\begin{lemma}\label{Lemma:GluingLeafInclusion}
For each $z$, $\tree_{\completion{V}}$ as above, $\im\leafInclusion_{z} \circ \Glue_{\completion{V}}^{\tree_{\completion{V}}}$ is contained in the image of some $\Glue_{\completion{W}}^{\tree_{\completion{W}}}$. Moreover, if $u \in \im \Glue_{\completion{W}}^{\tree_{\completion{W}}}$ is such that $\delbar u \in \transSubbundle^{\tree_{\completion{W}}}$ for some $\tree_{\completion{W}}$ with the image of $u$ not entirely contained in the half-cylindrical end $[0, \infty) \times \divSet$ of $W$, then $u \in \im\leafInclusion \circ \Glue_{\completion{V}}^{\tree_{\completion{V}}}$ for some $z$. 
\end{lemma}

\subsection{Vanishing of $\aug_{W}$}

Now we are ready to show that $\aug_{W}=0$ on orbits of action $\leq L$, using the perturbation scheme described above. To compute $\aug_{W}(\orbit_{\zeta})$ we need to count $\ind=0$ perturbed holomorphic planes $u$ in $\completion{W}$ which are positively asymptotic to $\orbit_{\zeta}$. Such $u$ live in the images of the $\Glue_{\completion{W}}^{\tree_{\completion{W}}}$, solving
\begin{equation*}
\delbar u = \multisec^{\tree_{\completion{W}}} = \multisec^{\tree_{\completion{W}}}_{\foliation} \oplus \multisec^{\tree_{\completion{W}}}_{\perp}, \quad \multisec^{\tree_{\completion{W}}}_{\foliation} \in \Sec\left(\im \Glue_{\completion{W}}^{\tree_{\completion{W}}}, \transSubbundle^{\tree_{\completion{W}}}_{\foliation}\right), \quad \multisec^{\tree_{\completion{W}}}_{\perp} \in \Sec\left(\im \Glue_{\completion{W}}^{\tree_{\completion{W}}}, \transSubbundle^{\tree_{\completion{W}}}_{\perp}\right).
\end{equation*}
Such solution spaces assumed transversely cut out and are oriented (after some choices have been made) by the usual contact homology orientation scheme. Since we are working in the $\ind=0$ case, this results in a rationally weighted collection of points, yielding a count in $\Q$.

Consider a $T \in [0, 1]$ family of sections $\multisec^{\tree_{\completion{W}}, T} = \multisec^{\tree_{\completion{W}}, T} = \multisec^{\tree_{\completion{W}}, T}_{\foliation} \oplus \multisec^{\tree_{\completion{W}}, T}_{\perp}$ subject to the following conditions:
\be
\item $\multisec^{\tree_{\completion{W}}, 0} = \multisec^{\tree_{\completion{W}}}$ and $\multisec^{\tree_{\completion{W}}, 1}_{\perp} = 0$ so that $\multisec^{\tree_{\completion{W}}, 1} = \multisec^{\tree_{\completion{W}}}_{\foliation}$.
\item The multisections are compatible over the overlaps of the $\im \Glue_{\completion{W}}^{\tree_{\completion{W}}}$.
\item The solution spaces $\{ \delbar u^{T} = \multisec^{\tree_{\completion{W}}, T}\} \subset \im \Glue_{\completion{W}}^{\tree_{\completion{W}}} \times [0, 1]$ are transversely cut out.
\ee
By Lemma \ref{Lemma:GluingLeafInclusion}, solutions to $\delbar u = \multisec^{\tree_{\completion{W}}, 1}$ are contained in leaves $\leafInclusion_{z}$. Indeed, by $\R_{s}$-invariance of almost complex structures on perturbations on the half-cylindrical end of $\completion{W}$ any such $\ind=0$ solution cannot be entirely contained in $[0, \infty) \times \divSet$. By the construction of multisections in \S \ref{Sec:MultisectionsInterior} -- in particular, invariance with respect to $z$ -- we have that at such a solution, $\multisec^{\tree_{\completion{W}}}_{\foliation}$ is determined by a $\multisec^{\tree_{\completion{V}}}$ so that $\Dlinearized_{\foliation_{W}}$ is surjective. Moreover at such a solution $\Dlinearized_{\perp}$ is automatically transverse by Equation \eqref{Eq:NormalPlaneAutomaticTransversality}. So we have transversality for the $T=0,1$ solution spaces, meaning that such a $\multisec^{\tree_{\completion{W}}, T}$ family of perturbations exists as in the third item above.

So by the construction of the $\multisec_{\foliation}$ solutions to $\delbar u = \multisec^{\tree_{\completion{W}}, 1}$ are of the form $\leafInclusion_{z}u_{\completion{V}}$ with $z \in \compactSubset_{\zeta, W}$, $u_{\completion{V}} \in \im \Glue_{\completion{V}}^{\tree_{\completion{V}}}$, and $\delbar u_{\completion{V}}= \multisec^{\tree_{\completion{V}}}$. Since $\ind u = 0$, it follows that $\ind u_{\completion{V}} = -\ind \Dlinearized_{\perp} = -\ind_{\Morse}\zeta \leq -1$. This violates surjectivity of $\Dlinearized_{u_{\completion{V}}}$.

Therefore the $T=1$ solution spaces are empty. Therefore for $T$ close to $1$ it follows that the solution spaces are empty as well. For the same reason it follows that by taking the $\multisec^{\tree_{\completion{W}}}_{\perp}$ to be sufficiently small, we can guarantee that the solution spaces are empty for all time. This guarantees that $\aug_{W}(\orbit_{\zeta})=0$.

\section{Non-vanishing of contact homology}\label{Sec:NonVanishing}

In this section we prove that the cover $M \times \Circle \times \R $ of a $M \times \torus^{2}$ Bourgeois contact manifold has non-vanishing contact homology. Following the arguments of \S \ref{Sec:IntroOutline}, this will complete our proof of Theorem \ref{thm:main}. To start, we review some background results from \cite{Avdek:Hypersurface}.

\subsection{The algebraic Giroux criterion}

Let $(W_+,\beta_+)$ and $(W_-,\beta_-)$ be two Liouville domains and $\phi:(\partial W_+, \beta_+|_{\partial W_+})\to (\partial W_-, \beta_-|_{\partial W_-})$ be a strict contactomorphism, meaning $\phi^{\ast}\beta_{-} = \beta_{+}$.

We use this data to define a convex hypersurface
\begin{equation*}
\hypersurface = (W_-\cup ([-1,1]_p\times \partial W_-) \cup_{\phi} W_+)
\end{equation*}
whose neighborhood $N(\hypersurface) = \R_{\tau} \times \hypersurface$ is equipped with a $\R_{\tau}$-invariant contact form $f\rd \tau+\beta$ as follows: Let $f$ be a function on $\hypersurface$ which
\be
\item is $\pm 1$ on $W_{\pm}$,
\item is an increasing function depending on $p$ along $[-1,1]_p\times \partial V_-$, and
\item is such that $f(0)=0$ with $\frac{\partial f}{\partial p}(0) > 0$.
\ee
Define $\beta$ by the properties that
\be
\item $\beta = \pm \beta_{\pm}$ along $W_{\pm}$ and
\item $h\beta_-|_{\partial W_-}$ along $[-1,1]_p\times \partial W_-$, where $h$ is a function of $p$ for which $h(p)\ge 1$, there is exactly one critical point at $0$ which is a maximum, $h(p)=2+p$ near $p=-1$, and $h(p)=2-p$ near $p=1$.
\ee
The gluing identifies the Liouville vector of $W_-$ with $-p\partial_p$ near $p=-1$ and the Liouville vector on $W_+$ with $-p\partial_p$ near $p=1$. Then $N(\hypersurface)$ is the neighborhood of convex surface $\tau=0$ considered in \cite{Avdek:Hypersurface}. We write $\xi = \ker(f\rd \tau+\beta)$ for the contact structure on $N(\hypersurface)$.

Given two Louville domains $W_{\pm}$, we have two DG augmentations $\aug_{\pm}:CC_*(\partial W_{\pm})\to \Q$ where $CC_{\ast}$ is the chain-level contact homology algebra. The definition of the algebra, of course, depends on auxiliary data including the choice of contact forms, almost complex structures on the sympelctizations of $\partial W_{\pm}$, and perturbations. The strict contactomorphism $\phi$ induces an is isomorphism $\phi_*:CC_*(\partial V_+)\to CC_*(\partial V_-)$ as long as the auxiliary data in the construction of the contact homology algebra is preserved by $\phi$. The following theorem {\cite[Theorem 1.1]{Avdek:Hypersurface}} provides an algebraic generalization of Giroux's criterion \cite[Theorem 4.5]{Giroux01}, which completely determines the tightness or overtwistedness of neighborhoods of convex hypersurfaces in contact $3$-manifolds.

\begin{theorem}\label{thm:AGC}
$CH(N(\hypersurface),\xi)\ne 0$ if and only if $\aug_+$ is DG homotopic (in the sense of \cite[Definition 14.5.1]{Avdek:Hypersurface}) to $\aug_-\circ \phi_*$.
\end{theorem}

\begin{rmk}
    The original statement of \cite[Theorem 1.1]{Avdek:Hypersurface} did not involve the strict contactomorphism $\phi$, but implicitly used the identity map as the gluing map. Strictly speaking, the identity map is only rigorously defined when $W_+=W_-$ as strict Liouville domains and we should always specify the gluing map $\phi$ in the general case. \Cref{thm:AGC} just clarifies the appearance of $\phi$, it has no mathematical modification to \cite[Theorem 1.1]{Avdek:Hypersurface}. 
\end{rmk}

\subsection{The Bourgeois case}

The augmentation $\aug_{W}$ of Theorem \ref{thm:no_augmentation} is determined by the contact hypersurface $\divSet \subset \completion{W} = \completion{V} \times \completion{\Sigma}$ with $\Sigma = [-1, 1]_{p}\times \Circle_{q}$ an annulus, choices of almost complex structures, and Kuranishi data. We recall that $\divSet = \divSet_{\Sigma} \cup \divSet_{V} = \partial W$ with $\divSet_{\Sigma} \simeq Y \times \Sigma$ containing all Reeb orbits which are contractible in $W$. Here we recall $Y=\partial V$. The subset $\divSet_{V}$ is a copy of $V\times \Circle$ for each boundary component of $\Sigma$ and so we can write $\divSet_V = (V \times \Circle)_{-1} \sqcup (V \times \Circle)_{1}$ with the connected components corresponding to the $p=\pm 1$ connected components of $\partial \Sigma$.

Recall that in the language of Proposition \ref{prop:BO},
\begin{equation*}
    (\Circle_{\tau} \times \hypersurface, \xi) = \BO(V, \phi) = \DG(W, W, \psi_{\BO}).
\end{equation*}
Here $\DG(W, W, \psi_{\BO})$ is the $\Circle$-invariant contact structure on $\Circle \times \hypersurface$ where $\hypersurface$ is the ``Bourgeois convex hypersurface whose positive and negative regions are both $W$, identified using $\psi_{\BO}: \divSet \rightarrow \divSet$. As described in Proposition \ref{prop:BO}, our gluing map $\psi_{\BO}$ for the $\hypersurface$ is the identity along $\divSet_{\Sigma} \cup (V \times \Circle)_{-1}$ and is given by $\phi \times \Id_{\Circle}$ along $(V \times \Circle)_{1}$. We recall that $\hypersurface$ is diffeomorphic to $M \times \Circle$, the $\R_{\tau}$ invariant contact structure on $\R_{\tau} \times \hypersurface$ covers the Bourgeois contact structure on $M \times \torus^{2}$, and that $\phi \in \Symp^{c}(V, \beta_{V})$ is the monodromy of the open book associated to $\Mxi$.

Equip $W_{+} = W$ with the data $\mathfrak{D}_{+}$ of the contact form $\alpha$ along $\divSet = \partial W_{+}$ together with the almost complex structures and Kuranishi data using which we computed $\aug_{+}=\aug_{W}$. \Cref{thm:no_augmentation} tells us that $\aug_{+}$ annihilates every good Reeb orbit generator of the contact homology algebra. We say that such an augmentation is \emph{trivial}\footnote{This is not standard but just for convenience, as the trivial map is not necessarily an augmentation in general. In particular, the trivial augmentation here is by no sense canonical in the general theory.}. Equip $W_{-}=W$ with the data $\mathfrak{D}_{-}$ of $\psi_{\BO}^{\ast}\alpha$ together with some choices of almost complex structures and Kuranishi data compatible with the data pulled back via $\psi_{\BO}$ along $\partial W_{-}$. We seek to show that the induced augmentation $\aug_{-}$ determined by $\mathfrak{D}_{-}$ is DG homotopic to the trivial augmentation.

We have $\psi_{\BO}^{\ast}\alpha=\alpha$ along $\divSet_{\Sigma} \cup (V \times \Circle)_{-1}$ and along $(V \times \Circle)_{1}$ where $\alpha = \rd q + \beta_{V}$ we have $\psi_{\BO}^{\ast}\alpha = \rd q + \beta_{V} + \eta$ for a closed $1$-form supported away from $\partial V$.

\begin{lemma}\label{lem:Gray}
    Let $\eta$ be a closed $1$-form on a Liouville domain $(V,\beta_{V})$ with support disjoint from $\partial V$. Then there is a $t\in [0,1]$ family of compactly supported isotopies $\delta_t:V \times \Circle_{q} \to V\times \Circle_{q}$ such that $\delta_t^*(\rd q + \beta_{V}+t\eta)=\rd q + \beta_{V}$, $\phi_0=\Id$, and $\delta_{t}$ is the identity in a neighborhood of $\partial V \times \Circle$ for all $t$.
\end{lemma}

\begin{proof}
    We write $\alpha_t=\lambda+t\eta+\rd q$, which is a contact form.   We assume $\phi_t$ is generated from integrating a $t$-dependent vector field $X_t$. Then $\phi_t^*\alpha_t=\alpha_0$ is equivalent to 
    $$\eta+L_{X_t}\alpha_t=0$$
    which, by Cartan's magic formula, is
    $$\eta+\iota_{X_t}(\rd \alpha_t)+\rd \alpha_t(X_t)=0.$$
    If we assume $X_t\in \xi_t=\ker \alpha_t$, then we need $\eta+\iota_{X_t}(\rd \alpha_t)=0$. Note that $\partial_q$ is the Reeb vector field for any $\alpha_t$. The equation $\eta+\iota_{X_t}(\rd \alpha_t)=0$ holds in $\partial_q$ direction tautologically. Then such $X_t\in \xi_t$ can be solved from $\eta+\iota_{X_t}(\rd \alpha_t)=0$ as $\rd \alpha_t$ is non-degenerate on $\xi_t$. That $\eta$ is compactly supported implies that $\phi_t$ is compactly supported.
\end{proof}

\begin{lemma}\label{Lemma:Double}
Let $(N, \xi)$ be a neighborhood of a convex hypersurface given by gluing the boundary of a Liouville domain $(W, \beta_{W})$ to itself using $\Id_{\partial W}$. In other words, this convex hypersurface is a \emph{symmetric double}. Then $CH(N, \xi)$ is non-zero.
\end{lemma}

\begin{proof}
We can use exactly the same choices of contact homology data to determine the augmentations associated to the positive and negative regions of this convex hypersurface. Hence $CH(N, \xi)\neq 0$ by Theorem \ref{thm:AGC}.
\end{proof}

\begin{prop}\label{prop:null}
The augmentation $\aug_{-}$ determined by $\mathfrak{D}_{-}$ is DG homotopic to the trivial augmentation.
\end{prop}

\begin{proof}
Let $\psi_{\delta_{t}}$ be the map $\divSet \rightarrow \divSet$ given by the identity along $\divSet_{\Sigma} \cup (V \times \Circle)_{-}$ and $\delta_{t}$ along $(V \times \Circle)_{+}$ where $\delta_{t}$ is as in Lemma \ref{lem:Gray}. Then $\psi_{\delta_{1}}^{\ast}\alpha = \psi_{\BO}^{\ast}\alpha$ so we can form a convex hypersurface $\hypersurface_{\delta}$ by gluing the boundaries of the $W_{-}$ to $W=W_{+}$ together using $\psi_{\delta_{1}}$. Since this convex hypersurface is determined only by the contactomorphism, Lemma \ref{lem:Gray} tells us that $\hypersurface_{\delta}$ is a symmetric double. Since we can use the contact homology data $\mathfrak{D}_{\pm}$ to compute the augmentations associated to the positive and negative regions of $\hypersurface_{\delta}$, Theorem \ref{thm:AGC} and Lemma \ref{Lemma:Double} combine to inform us that $\aug_{-}$ computed using $\mathfrak{D}_{-}$ is DG homotopic to $\aug_{+}$ computing using $\mathfrak{D}_{+}$. Since this $\aug_{+}$ is a trivial augmentation, the proof is complete.
\end{proof}

\begin{proof}[Proof of \Cref{thm:main}]
By \Cref{prop:null}, we have $\aug_{-}\circ (\psi_{\BO})_*$ is DG homotopic to the trivial augmentation. Therefore \Cref{thm:AGC} implies that a $\R_{\tau}$-invariant neighborhood of the Bourgeois convex hypersurface has non-vanishing contact homology. Since this $\R_{\tau}$-invariant neighborhood covers $\BO(V,\phi)$, we conclude that $\BO(V,\phi)$ is tight.
\end{proof}

\bibliographystyle{alpha}      
\bibliography{ref}

\bigskip
\footnotesize

Russell Avdek, \par\nopagebreak
\textsc{Laboratoire de Mathématiques d'Orsay, Universit\'{e} Paris-Saclay, France}\par\nopagebreak
\textit{E-mail address:} \href{mailto:russell.avdek@gmail.com}{russell.avdek@gmail.com}

\bigskip

Zhengyi Zhou, \par\nopagebreak
\textsc{Morningside Center of Mathematics, Chinese Academy of Sciences;}\par\nopagebreak
\textsc{Academy of Mathematics and Systems Science, Chinese Academy of Sciences, China}\par\nopagebreak
\textit{E-mail address}: \href{mailto:zhyzhou@amss.ac.cn}{zhyzhou@amss.ac.cn}

\end{document}